\documentclass[a4paper,12pt]{article}
\usepackage[utf8]{inputenc}
\usepackage[english]{babel}
\usepackage{amsthm}
\usepackage{amssymb}
\usepackage{amsmath}
\usepackage{mathrsfs}
\usepackage{amssymb}
\usepackage{pdfsync}
\usepackage{graphicx, color}

\newtheorem{thm}{Theorem}[section]
\newtheorem{definition}[thm]{Definition}
\newtheorem{prop}[thm]{Proposition}
\newtheorem{lemma}[thm]{Lemma}
\newtheorem{example}[thm]{Example}
\newtheorem{rmk}[thm]{Remark}
\title{Homogenization of metrics in oscillating manifolds}
\author{Andrea Braides \\Dipartimento di Matematica, Universit\`a di Roma Tor Vergata
\\ via della ricerca scientifica 1, 00133 Roma, Italy\\ \\ Andrea Cancedda and Valeria Chiad\`o Piat\\ Dipartimento di Matematica, Politecnico di Torino \\  corso Duca degli Abruzzi 24, 10129 Torino, Italy
}

\date{}                                           

\def\e{\varepsilon}

\newcommand{\Rn}[1]{\mathbb{R}^{#1}}
\newcommand{\eps}{\epsilon}
\newcommand{\gammalim}[1]{\Gamma\hbox{-}\lim_{\eps \to 0}\ {#1}}

\newcommand{\weakconv}{\rightharpoonup}
\newcommand{\ueps}{u_{\eps}}
\newcommand{\veps}{v_{\eps}}

\newcommand{\Feps}{F_{\eps}}

\newcommand{\psihom}{\psi_{{\rm hom}}}

\def\eps{\varepsilon}

\begin{document}


\maketitle
\abstract{\noindent We consider energies defined as the Dirichlet integral
of curves taking values in fast-oscillating manifolds converging
to a linear subspace.
We model such manifolds as subsets of $\Rn{m+m'}$
described by a constraint $(x_{m+1},\ldots,x_{m'})= \delta \,
\varphi(x_1/\e,\ldots, x_m/\e)$ where $\e$ is the period of the oscillation,
$\delta$ its amplitude and $\varphi$ its profile. The interesting case
is $\e<\!<\delta<\!<1$, in which the limit of the energies
is described by a Finsler metric on $\Rn{m}$ which is defined by
optimizing the contribution of oscillations on each level set $\{\varphi=c\}$.
The formulas describing the limit mix homogenization and convexification
processes, highlighting a multi-scale behaviour of optimal sequences.
We apply these formulas to show that we may obtain all (homogeneous)
symmetric Finsler metrics larger than the Euclidean metric as limits in 
the case of oscillating surfaces in $\Rn{3}$.}

\section{Introduction}\label{sec intro}
The object of this paper is the asymptotic analysis of integral problems
with oscillating constraints.  There is a wide literature concerning homogenization problems for singular structures and for functions defined on networks and periodic manifolds (see, e.g., \cite{A-B-CP,BC,BF,P-Z06,R04,Z}). In most of those problems the geometric complexity is in the domain of definition, and the functions are considered as traces of functions defined on the whole space or as limits of functions defined on full-dimensional sets as those sets tend to a lower dimensional (possibly multidimensional) structure. In our case we take into account similar geometries, but the geometrical complexity is in the codomain, as we consider instead functions with values in a periodic manifold, and we analyze the behavior of the corresponding energies as the geometry of the target manifold gets increasingly oscillating. Our results focus on the behaviour of energies defined on functions constrained to take their values on manifolds $V_{\eps}$ with a finely oscillating geometry, as these manifolds converge to a smoother manifold $V$ as $\eps \to 0$.
Homogenization problems with a fixed target manifold $V$ have been considered in \cite{BM2, BM}.

Since we are interested in highlighting the effects of the constraint, we will focus on a prototypical energy functional; i.e, the Dirichlet integral. Namely, for $u:\Omega \subset \Rn{n} \to V_{\eps}$ we will consider
$$
\Feps(u)=	\begin{cases}
      		\displaystyle \int_{\Omega} |\nabla u|^2 dx	& u \in H^1(\Omega;V_{\eps}), \cr
      		+\infty					& \text{otherwise}.
\end{cases}
$$
We suppose that the limit $V$ is a smooth $m$-dimensional manifold and the oscillating $V_\e$ are manifolds of the same dimension lying in a tubular neighbourhood of $V$ with vanishing radius as $\e\to0$. The asymptotic description will be given in terms of the computation of the $\Gamma$-limit of $\Feps$ \cite{GCB,DM,HB}.

We will treat the {\em cartesian case}\/; i.e., when the manifolds $V_\e$ can be seen as
graphs of functions defined on $V$ (the latter identified with $\Rn{m}$). More precisely,
we suppose that there exist functions $\varphi_{\eps}: \Rn{m} \to \Rn{m'}$ such that
\begin{equation}\label{cartes}
V_{\eps}=\left\{ (x,\varphi_{\eps}(x)): x \in \Rn{m} \right\} \subseteq \Rn{m+m'}.
\end{equation}
Hence the assumption that $V_{\eps}$ converges to $V$ as $\eps \to 0$ is translated into
$$\lim_{\eps \to 0} \|\varphi_\eps\|_\infty = 0.
$$
This description can be thought as a local picture of the more general case, where $V$ is not necessarily a hyperplane and by a localization and blow-up argument we can consider the tangent space to $V$ at some point $X$ in place of $V$ in the model that we analyze.

Our modeling assumption is that the description of the oscillations of $\varphi_\e$ is obtained through a single periodic function $\varphi: \Rn{m} \to \Rn{m'}$ satisfying
\begin{enumerate}
  \item $\varphi: \Rn{m} \to \Rn{m'}$ is $(0,1)^m $-periodic;
  \item $\displaystyle\varphi_\eps(x)=\delta\, \varphi\Bigl({x\over\eps}\Bigr)$, with $\delta=\delta_\eps \to 0$ as $\eps \to 0$.
\end{enumerate}

In such a setting a function $u \in H^1(\Omega;V_{\eps})$ can be rewritten as
\[
u(x)=(u_1(x),u_2(x)),
\]
with $u_1:\Omega \to \Rn{m}$ and
$$
u_2(x)= \varphi_\eps(u_1(x)).
$$
Hence, we can write $\Feps$ without the constraint $u\in V_\e$, in terms of $u_1$ as
\begin{eqnarray*}
\Feps(u) &= &\int_{\Omega} |\nabla u|^2 dx = \int_{\Omega} |\nabla u_1|^2 dx + \int_{\Omega} |\nabla \varphi_{\eps} (u_1)|^2 dx
\\
&=&\int_{\Omega} |\nabla u_1|^2 dx + \left( \frac{\delta}{\eps}\right)^2\int_{\Omega} |\nabla \varphi\Bigl({u_1\over\e}\Bigr) \nabla u_1|^2 dx.
\end{eqnarray*}

The coefficient in front of the second term in this last expression suggests three different behaviors depending on the scale of the coefficient $\delta$:

\begin{enumerate}
  \item $\delta/\eps \to 0$. In this case the homogenization becomes trivial, the second term can be neglected and the $\Gamma$-limit is just the Dirichlet integral of the function $u_1$, which in particular is independent of the constraint and the function $\varphi$;

  \item $\delta/\eps \to c\in (0,+\infty)$. In this case by a comparison argument we can actually suppose that $\delta/\eps=c$ and consider the energy density
$$
f_c(v,\xi)=  |\xi|^2 + c^2 |\nabla \varphi (v ) \xi|^2
$$
so that, with a slight abuse of notation,
$$
\Feps(u)= \Feps(u_1)=\int_{\Omega} f_c\Bigl( {u_1\over\e},\nabla u_1\Bigr)dx.
$$
Since $f_c$ is periodic and satisfies a standard growth condition
the homogenization of these energies can be then performed
by using general almost-periodic homogenization theorems (see \cite{BDF} Chapter 15);

  \item $\delta/\eps \to +\infty$. This is the new and interesting case when the energy  density of $F_\e$ does not satisfy standard growth conditions and we cannot use known results. The fact that the coefficient of the second term blows up as $\eps \to 0$, suggests that the behavior of the homogenized functional is related to conditions that make the second integral negligible as $\e\to 0$. Upon scaling the variable, this leads to the condition that $u_1$ make $\nabla_y \varphi (u_{1})$ almost zero; i.e. $u_1$ is very close to lying on a level set of $\varphi$ (at least locally). Therefore, the $\Gamma$-limit will strongly depend on the geometry of the constraint, in particular on the level sets of $\varphi$.
\end{enumerate}

The results in this paper deal with the description of asymptotic {\em metric properties} of $V_\e$,
for which we deal with {\em curves} in $\Rn{m}$ (i.e., $n=1$). The general vectorial case $n>1$ seems to include additional effects than in the case of curves, which require the use of notions as quasiconvexity, polyconvexity and rank-1-convexity, and is beyond the scope of this work. By a sectional argument, however, the results with $n=1$ provide a lower bound for the case $n>1$.

\bigskip
In order to describe the $\Gamma$-limit in the case $\delta>\!>\e$ we have to introduce several types of homogenization. With a slight abuse of notation from now on we will directly use the variable $u$ in place of $u_1$.

First, with fixed $z$ we consider the {\em strict constraint}
\begin{equation}\label{stricco}
\varphi\Bigl({u\over \e}\Bigr)=z,
\end{equation}
which is meaningful only if $z$ is in the image of $\varphi$.
For functions satisfying this constraint the functional $F_\e(u)$ reduces to the Dirichlet integral.
Moreover, a limit of functions satisfying this constraint is not constant only if the set
\begin{equation}
L_{\varphi}^z=\{x\in \Rn{m}: \varphi(x)=z\}
\end{equation}
 contains an unbounded connected curve with locally finite length.
We will consider the (slightly) stronger assumption:

\smallskip
$\bullet$ ({\bf uniform connectedness}) there is a single unbounded connected component
of this set and all pairs of points $x$, $x'$ in this component can be connected with a curve  of length proportional to the distance between $x$ and $x'$ lying in the component.

We will specify separately these topological and metric properties of the level sets in Definitions \ref{esistenza raccordi} and \ref{HyH}.

\smallskip
This property is easily verified if 
$L_{\varphi}^z$ can be written as the union of sets $\{x\in \Rn{m}: \varphi_j(x)=z_j\}$
which are composed of unions of periodic $C^1$ hypersurfaces, for some $\varphi_j$ and $z_j$.

\smallskip
Under the uniform-connectedness assumption, the homogenization of the Dirichlet integral with the strict constraint (\ref{stricco}) is described by an integral functional
$$
\int_\Omega \psi^z_{\rm hom} (u')\,dt,
$$
with $\psi^z_{\rm hom}:\Rn{m}\to[0,+\infty)$ a two-homogeneous convex function.
Moreover, $\psi^z_{\rm hom}$ satisfies the {\em asymptotic homogenization formula}
$$
\psi^z_{\rm hom}(w)=\lim_{T\to+\infty} \psi^z_T(w),
$$
where
\begin{eqnarray}\label{def psiTz intro}
\psi_T^{z} (w)=  \frac{1}{T} \min \Bigl\{  \int_0^T |u'|^2dt:   \varphi(u)=z,
 |u(0)|\le  \sqrt{m},|u(T)-Tw| \le  \sqrt{m} \Bigr\}.
\end{eqnarray}
This is a variation on corresponding formulas for the homogenization of functionals with energy densities $f(u/\e,u')$ (see [8], Chapter 15). Note that, in the present case such energy densities are defined with the value $+\infty$ outside the constraint, so that some extra care must be taken; in particular, we cannot easily impose strict boundary conditions (that in the usual case would read $u(0)=0$ and $u(T)=Tw$). We prefer to substitute those conditions with an inequality, which in particular is satisfied when the initial datum $u_0=u(0)$ is any point in the periodicity cube $(0,1)^m$ satisfying $\varphi(u_0)=z$.

\smallskip
In order to derive the homogenization theorem for our original energies $F_\e$ from the strictly constrained energies we make the following assumptions:

$\bullet$ for all $z$ in the image of $\varphi$, either $L^z_\varphi$ is uniformly connected (in the sense defined above) or composed by disjoint bounded closed components;

$\bullet$ curves satisfying the {\em weaker constraint} $u(t)\in\{x\in \Rn{m}: |\varphi(x)-z|\le c\}$
almost everywhere are close to curves satisfying a strict constraint for some $z'$ whenever $c$ is small enough.

Before stating the latter condition more precisely, we consider the model example of $m=2$
and
$$
\varphi(x,y)=\sin(2\pi x)\sin (2\pi y).
$$
In this case, the only connected level set $L^z_\varphi$ is with $z=0$. A set $\{(x,y): |\varphi(x,y)-z|\le c\}$, with $z\neq 0$,
is either composed of disconnected components (when $c<|z|$), or contains a tubular neighbourhood of $L^0_\varphi$. In any case, given a curve $u$ taking values in that set, we can find a curve $v$ satisfying the strict condition $\varphi(v(t))=0$ close to the original curve and with energy not greater than the energy of $u$ times $1+o{(}1)$ as $c\to 0$.

With this example in mind we can state the second condition above as follows.
For $w \in \Rn{m}$, $z \in{\rm Im}(\varphi) \subset \mathbb{R}$, $c,T>0$, we consider the minimum problems
\begin{eqnarray}\label{def psiTzc intro}\nonumber
\psi_T^{z,c} (w)&= & \frac{1}{T} \min  \Bigl\{\int_0^T |u'|^2dt :  |\varphi(u)-z| \le  c,\\
&&\qquad  |u(0)|\le  \sqrt{m} , \  |u(T)-Tw|\le  \sqrt{m} \Bigr\}.
\end{eqnarray}
Then we require that there exists $z'$ such that $|z-z'|\le c$ and for all $w$ there exists $w'=w+o_T(1)$ such that
$$
\psi_T^{z,c} (w)\ge (1+o_c(1))\, \psi_T^{z'} (w')+ o_T(1),
$$
where $o_c(1)\to 0$ as $c\to 0$ and $o_T(1)\to 0$ as $T\to+\infty$.

This is the (rather complex) variational formulation of a {\em geometric stability property} of level sets which is easily proved for ordinary constraints. By using this property it is possible to prove the homogenization result by reducing to functions satisfying strict constraints.
We summarize the main arguments in the
proof: from energy bounds we deduce that functions $u_\e$ with equibounded $F_\e(u_\e)$ locally must lie in some set $\{x:|\varphi(u_\e/\e)-z|\le c\}$
with $c$ small. By a scaling argument then the energy is estimated using $\psi_T^{z,c} (w)$, where $w$ is the local averaged slope of $u_\e$,
and eventually with $\psi_T^{z'} (w')$. A particular care has to be taken in these
computations in order to reduce to ``local'' estimates which nevertheless, after scaling, can be estimated by problems with $T$ large enough.
Finally, we get our homogenized energy density by optimizing over $z$ and over mesoscopic oscillations between level sets, producing a convex envelope of the minima between homogenization formulas on strict constraint:
\[
\psihom= \Bigl( \min_{z \in {\rm Im}(\varphi)} \psihom^z\Bigr)^{**}.
\]
Note that under our assumptions on level sets, in dimension two there is only one infinite connected level set for some $z=z_0$, so that this formula simplifies to $\psihom^{z_0}$.
In dimension three or higher it is instead possible to give examples where there are more
than one infinite connected level set, and the formula above must indeed be applied.
Eventually, the $\Gamma$-limit is given by
\[
\Gamma\hbox{-}\lim_{\eps\to 0} \Feps(\ueps) = \int_0^1 \psihom(u') dt.
\]
We remark that the proof of the upper bound does not rely only on the construction of periodic recovery sequences
from minimizers of $\psi^z_T$ as customary in homogenization problems, but also on a multi-scale argument necessary to connect such sequences for different $z$.

The functionals $\Feps$ in the scalar case $m^\prime=1$ can be geometrically interpreted as follows: a level set of $\varphi$ containing a unique unbounded connected component can be viewed as a sort of periodic unbounded and connected $\eps$-network over $\Rn{m}$, that represents the ``allowed'' zones for curves $\ueps$ in sublevel sets of $\Feps$. Indeed, if $\ueps$ lie on this network, the gradient of $\varphi(u_\eps/\eps)$ will be zero and the $\Gamma$-limit can be finite on their limit $u$. In this metric standpoint, we can interpret $\psihom$  as measuring the distance between the origin and the point $w$, not with the euclidean norm, but with the length of a curve that microscopically lies in the lattice defined by the constraint.

\smallskip
The plan of the paper is as follows: Section \ref{chapter to infty} is devoted to the statement of the problem, the assumptions, and the statement of the main results (Theorem \ref{thm gammalim}). Preliminary results are in Section \ref{sec prelim} and Theorem \ref{thm gammalim} is proved in Section \ref{sec gamma lim}. Some examples follow in Section \ref{esempi}. In the final section of the paper, when $m=2$, we apply our result to  the characterization of all metrics on $\Rn{2}$ that can be obtained from some  $\varphi$, as a $\Gamma$-limit, following the procedure above. This problem is linked to the description of homogenized Riemannian metrics by Braides, Buttazzo and Fragal\`a \cite{B-B-F} (see also Burago \cite{Bu}).
In our case all limit metrics $\psihom$ will satisfy the necessary condition $\psi(z)\ge|w|^2$, with the equality achieved when $\varphi=0$.
Indeed, we prove that this condition is also sufficient; that is, any (possibly degenerate) symmetric Finsler metric larger than the Euclidean metric can be approximated by a $\psihom$ defined through the homogenization process of an oscillating constraint problem.

\section{Statement of the homogenization result}\label{chapter to infty}



The function $\varphi : \Rn{m} \to \Rn{m'}$ will be a fixed $1$-periodic Lipschitz function.
For $\e>0$ and $\delta=\delta_\e>0$ we consider $F_{\eps}: L^2([0,1], \Rn{m}) \to [0,+\infty]$ in the following unconstrained form

\begin{equation}\label{def Feps curve}
F_{\eps}(u)= \int_{0}^1\Bigl( | u'|^2 + \Bigl(\frac{\delta}{\eps}\Bigr)^2 \Bigl| \nabla_y \varphi \Bigl({u\over\e}\Bigr)  u' \Bigr|^2\Bigr) dt.
\end{equation}


For all $w \in \Rn{m}, z \in{\rm Im}(\varphi), c>0$, we define the energy densities
\begin{eqnarray}\label{def psiTzc}
\psi_T^{z,c} (w)=  \frac{1}{T} \min  \Bigl\{ \int_0^T |u'|^2dt: \  |\varphi(u)-z| \leq c,
 |u(0)|\leq \sqrt{m} ,|u(T)-Tw|\leq \sqrt{m} \Bigr\},
\\
\label{def psiTz}
\psi_T^{z} (w)= \frac{1}{T} \min \Bigl\{\int_0^T |u'|^2dt: \varphi(u)=z ,
|u(0)|\leq\sqrt{m},|u(T)-Tw| \leq \sqrt{m}
\Bigr\}.\qquad\quad
\end{eqnarray}

\begin{rmk}\rm
Note that, for some constraints $\varphi$ and $z$, the minimum in equations (\ref{def psiTzc}) and (\ref{def psiTz}) may be performed on an empty set, depending on the choice of $c$, $T$ and $w$. In this case we set $\psi_T^{z,c} (w)= +\infty$ or $\psi_T^{z} (w)=  +\infty$, respectively.
\end{rmk}

Before getting to our main result, we state some definitions that will use be used to clarify the hypotheses on the constraint. We begin with two geometric hypotheses for the function $\varphi$ and for its level sets or, equivalently, for the function $\psi_T^{z,c}$.

\begin{definition}\label{hp geom} We say that the constraint function
satisfies the {\em geo\-met\-ric-stability property} if  there exists a continuous function $\omega(c)$, with $\omega(c) \to 0$ as $c \to 0$, such that, for any $T>0$, $w \in \Rn{m}$, $z \in{\rm Im}(\varphi)$, one of the following two conditions is satisfied
\begin{equation}
\psi^{z,c}_T(w) = +\infty;
\end{equation}
\begin{equation}\label{eq geom}
\psi_T^{z,c}(w) \ge  (1-\omega(c)) \psi_T^{z'}(w') - \frac{k(c)}{T},
\end{equation}
for suitable $z' \in{\rm Im}(\varphi)$, $w' \in \Rn{m}$ and $k(c) \in \Rn{}$
such that $|z-z'|\le  c$, $|w-w'| \le  \sqrt{m}/T$ and $k(c)$ is independent of z.
\end{definition}

\begin{definition}\label{esistenza raccordi} The function $\varphi$ satisfies a {\em controlled-length condition} if for all $z \in{\rm Im}(\varphi)$ such that $L^z_\varphi $ is connected, there exist a constant $C \in \Rn{}$ such that for any $x,y \in \left\{\varphi = z\right\}$ there exists
a path $\gamma:[0,1] \to \Rn{m}$,$\gamma \in H^1([0,1])$, with $\gamma(0)=x$, $\gamma(1)=y$ and $\varphi(\gamma(t))=z$, such that
\[
l(\gamma):=\int_0^1|\gamma'(t)| dt \le  C |x-y|.
\]
%
\end{definition}

Note that we can prove that $\varphi$ satisfies Definition \ref{esistenza raccordi} assuming that if $L^z_\varphi$ has a connected unbounded component then it is the locally finite union of $\mathcal{C}^1$ sets.

\begin{definition}\label{HyH}
We say that $\varphi$ {\em has non-degenerate levels} if
for any $z \in{\rm Im}(\varphi)$ one of the two following conditions holds true for the set $L^z_\varphi$:
\begin{eqnarray*}\nonumber
&{\rm i)}& \hbox{ it is composed of a unique unbounded connected component;}
\\
&{\rm ii)}& \hbox{ it contains no unbounded connected components};
\end{eqnarray*}
and there exists at least one level set $L^z_\varphi $ satisfying ${\rm (i)}$.
\end{definition}

If $\varphi$ satisfies the conditions above we will prove that the limit
\begin{equation}\label{psiT hom eq}
\psi_{\text{hom}}^z (w)= \lim_{T \to \infty} \psi_T^{z} (w),
\end{equation}
with $\psi_T^{z} (w)$ defined in (\ref{def psiTz}),
exists (in Lemma \ref{lemma ex limit}
below); we then define
\begin{equation}\label{eq psihom}
\psi_{\text{hom}}= \Bigl( \min_{z  \in{\rm Im}(\varphi)} \psi_{\text{hom}}^z \Bigr)^{**},
\end{equation}
where the double asterisk stands for the convex envelope.

\begin{rmk}\rm
Note that the hypothesis in Definition \ref{HyH} rules out the case when we have
some $z$ with infinitely many disjoint unbounded connected components
(e.g., $\varphi(x,y)=\sin x$ in $\Rn{2}$). Such cases present the technical difficulty that
$\psi_{\text{hom}}^z$ will be degenerate in some directions. We will separately deal
with these situation in two dimensions $m=2$ in Section \ref{Chapter Finsler metrics}.
\end{rmk}

Now we can finally state our main result.

\begin{thm}\label{thm gammalim}
 Let the constraint function $\varphi$
 satisfy the conditions in Definitions {\rm\ref{hp geom}}, {\rm\ref{esistenza raccordi}} and  {\rm\ref{HyH}}. Let $\Feps$ be defined in {\rm(\ref{def Feps curve})}; then, for any $u \in L^2([0,1];\Rn{m})$
\[
\gammalim{F_{\eps}(u)}=F(u)=	\begin{cases}\displaystyle
						\int_0^1 \psi_{\text{\rm hom}} (u') dt	&	u \in H^1([0,1];\Rn{m})\\
						+\infty					&	\text{otherwise},
						\end{cases}
\]
in the strong topology of $L^2([0,1];\Rn{m})$, where $\psi_{\text{\rm hom}}$  is defined by {\rm(\ref{psiT hom eq})} and {\rm (\ref{eq psihom})}.
\end{thm}

\begin{rmk}\rm
The hypotheses on $\varphi$ are not optimal.
We conjecture that the hypothesis of non-degenerate levels in Definition \ref{HyH}
may be dropped, in which case the limit $\psihom$ may take the value $+\infty$
outside a linear space. In this case however, the proof of the existence of
the homogenization formula in the next section does not hold.
Furthermore, it is not clear whether the controlled-length condition
in Definition \ref{esistenza raccordi} can be relaxed to only requiring that
each pair of points $x,y$ be connected by a curve of finite length.
\end{rmk}

\section{Preliminary results}\label{sec prelim}

Using the geometric hypotheses stated above, we first prove the homogenization formulas. The proof follows standard subadditive arguments, which work thanks to the
hypotheses in Definitions {\rm\ref{hp geom}},  {\rm\ref{esistenza raccordi}} and  {\rm\ref{HyH}}, but have to be followed with some additional care due to the presence of the constraint.

\begin{lemma}\label{lemma ex limit}
Let $\varphi$ satisfy the conditions in Definitions {\rm\ref{hp geom}},  {\rm\ref{esistenza raccordi}} and {\rm\ref{HyH}}. Let $w \in \Rn{m}$ and $z \in{\rm Im}(\varphi)$,
and for all $T>0$ let  $\psi^z_T(w)$ be given by {\rm(\ref{def psiTz})}.
Then, the limit
$$
\psi^z_{\hom}(w)=\lim_{T\to+\infty} \psi^z_T(w)
$$
exists. If the set $L^z_\varphi $ is composed of a unique unbounded
connected component then $\psi^z_{\hom}(w)<+\infty$ for all $w$.
\end{lemma}

\begin{proof}
If $z$ is such that the set $L^z_\varphi $ contains no unbounded connected component then if $w\neq 0$ we have $\psi^z_T(w)=+\infty$ for $T$ large enough,
while for $w=0$ we have $\psi^z_T(w)=0$ for all $T$, so that the limit trivially exists.
Hence, by the hypotheses of Definition \ref{HyH} we can suppose that the set $L^z_\varphi $ is composed of a unique unbounded connected component. In this case $\psi^z_T(w) < +\infty$ for any $w$ and $T$, thanks to Definition \ref{esistenza raccordi}; indeed, to check this we can suppose that $\varphi(0)=z$, and the level $L^z_\varphi $ connects $0$ with any vector of the basis $e_1,\ldots, e_m$.
Then test functions for $\psi^z_T$ can be constructed by concatenating translations of these connections.

Let $v_T:[0,T] \to \Rn{m}$ be a minimizer for $\psi_T^z(w)$.
For $S>T$ we want to construct a competitor for $\psi^z_S(w)$ using $v_T$ by a patchwork procedure. In what follows we denote $[Tw]\in{\mathbb Z}^m$ the integer part component-wise of $Tw \in \Rn{m}$.

Let $K=K_{S,T}=\bigl[{S\over T+1}\bigr]$.
We consider $K$ curves $\gamma_k:[0,1]\to \Rn{m}$
satisfying pointwise the constraint $\varphi(\gamma_k(t))=z$ for all $t$, such that
$$
\gamma_k(0)=v_T(T)+[(k-1)(T+1)w],\qquad \gamma_k(1)=v_T(0)+[k(T+1)w],
$$
for $k\in\{1,\ldots,K-1\}$, and
$$ \gamma_K(0)=v_T(0)+[(K-1)(T+1)w],
\quad
\gamma_K(1)=w_S \hbox{ with } |w_S-Sw|\le\sqrt m.
$$
Note that
\begin{eqnarray*}
|\gamma_k(1)-\gamma_k(0)|&\le&|v_T(0)|+|Tw-v_T(T)|+ |w|\\ &&
\qquad+|[k(T+1)w]-k(T+1)w|
\\ &&
\qquad+|[(k-1)(T+1)w]-(k-1)(T+1)w|
\\
&\le& 4\sqrt m+|w|,
\end{eqnarray*}
for $k\in\{1,\ldots, K-1\}$ and
$$|\gamma_K(1)-\gamma_K(0)|\le 6\sqrt m+|w|\sqrt m (T+3).$$
Hence, by the hypothesis in Definition \ref{esistenza raccordi}, such curves exist satisfying in addition
\begin{eqnarray*}
l(\gamma_k) \le  C_1\hbox{ for } k\in\{1,\ldots, K-1\},\qquad l(\gamma_K) \le   C_1(1+T)
\end{eqnarray*}
with $C_1$ depending only on $w$ and the dimension $m$. Up to a reparameterization, we can also assume that $\gamma_k$ have constant velocity, $|\gamma_k'|=l(\gamma_k)$. Now define the function $v_S(t):\left[0, S\right] \to \Rn{m}$ by

\begin{equation}\label{def vS}
v_S(t) = 	\begin{cases}\displaystyle
			v_{T,[{t\over T+1}]+1}\Bigl(t-\Bigl[{t\over T+1}\Bigr](T+1)\Bigl)
			\\&\hskip-3cm\hbox{ if }0\le  t \le (K-1)(T+1)\cr
			\gamma_K\Bigl({t-(K-1) (T+1)\over S-(K-1) (T+1)}\Bigr)\\
				&\hskip-3cm\hbox{ if } (K-1)(T+1)\le  t\le S,
			\end{cases}
\end{equation}
where $v_{T,k}:[0,T+1]\to \Rn{m}$ is defined by
\begin{equation}\label{def vT tilde}
v_{T,k}(t) = 	\begin{cases}\displaystyle
			v_T(t)		+[(k-1)(T+1)w]		&\hbox{ if }0\le  t \le T\\					
			\gamma_k(t-T)	&\hbox{ if } T \le  t\le T+1
			\end{cases}
\end{equation}
for $k\in\{1,\ldots,K\}$. Note that $v_{T,k}(T+1)= v_{T,k+1}(0)$ for $k\in\{1,\ldots,K-2\}$,
and that $v_{T,K-1}(T+1)= \gamma_K(0)$,
 so that $v_S$ is a continuous function.


By construction we have $|v_S(0)|=|v_T(0)| \le  \sqrt{m}$, $|v_S(S)-Sw|=|[Sw]-Sw| \le  \sqrt{m}$, and $\varphi(v_{S}(t))=z$, by the periodicity of $\varphi$. Therefore, we have
\begin{eqnarray}\nonumber
\psi_S^z(w) &\le & \frac{1}{S} \int_0^S |v'_S(t)|^2 dt \\
\nonumber
&= &
\frac{1}{S}\Bigl(\int_0^{(K-1)(T+1)}
|v'_S(t)|^2 dt
+\int^S_{(K-1)(T+1)}
|v'_S(t)|^2 dt\Bigr)
\\ \nonumber
&= &
\frac{1}{S}\Bigl(\sum_{k=1}^{K-1}\int_0^{T+1} | v'_{T,k}(t)|^2 dt
+\int^S_{(K-1)(T+1)} |v'_S(t)|^2 dt\Bigr)
\\ \nonumber
&= &
\frac{1}{S}\Biggl(\Bigl(\Bigl[{S\over T+1}\Bigr]-1\Bigr)\int_0^{T} |v'_T(t)|^2 dt
+\sum_{k=1}^{K-1}\int_0^1|\gamma_k'(t)|^2\,dt\\
&&\nonumber\qquad\qquad +
{1\over S-(K-1)(T+1)}\int_0^1 |\gamma_K'(t)|^2 dt\Biggr)
\\ \nonumber
&= &
\frac{1}{S}\Biggl(\Bigl(\Bigl[{S\over T+1}\Bigr]-1\Bigr)\int_0^{T} |v'_T(t)|^2 dt
\\ \nonumber
&&\qquad\qquad+\sum_{k=1}^{K-1}l(\gamma_k)^2+
{1\over S-(K-1)(T+1)}l(\gamma_K)^2\Biggr)
 \nonumber
\end{eqnarray}
\begin{eqnarray}\nonumber
&\le &
\frac{1}{S}\Biggl(\Bigl[{S\over T+1}\Bigr](T\psi^z_T(w)+C_1)+{1\over (T+1)}l(\gamma_0)^2\Biggr)
\\
 \nonumber
&\le &
\frac{1}{S}\Bigl(\Bigl[{S\over T+1}\Bigr](T\psi^z_T(w)+C_1)+C_1(1+T)\Bigr)
\\
&\le &
\psi^z_T(w)+{C_1\over T}+C_1{1+T\over S}.\label{sthom}
\end{eqnarray}

Now, taking first the limsup as $S \to \infty$ and then the liminf as $T \to \infty$, we get
$$
\limsup_{S \to \infty} \psi_S^z(w) \le  \liminf_{T \to \infty} \psi_T^z(w),
$$
and the existence of the limit $\psihom^z(w)<+\infty$.
\end{proof}

\begin{rmk}\label{unist}\rm
If the set $L^z_\varphi $ is composed of a unique unbounded connected component then from (\ref{sthom}) we have, passing to the limit as $S\to+\infty$,
\begin{equation}\label{stisti}
\psi_{\hom}^z(w) \le  \psi_T^z(w) +{C_1\over T}
\end{equation}
for all $w$ and $T$.
\end{rmk}

For the function $\psihom^z(w)$ the following property holds:

\begin{prop}\label{prop psihom 2om}\ Let $\varphi$ satisfy the conditions in Definitions {\rm\ref{hp geom}},  {\rm\ref{esistenza raccordi}} and {\rm\ref{HyH}}. For all $z$ the function
$\psihom^z$ is a $2$-homogeneous function; i.e., for any $\lambda \neq 0$ and $w \in \Rn{2}$ one has
$$
\psihom^z(\lambda w) = \lambda^2 \psihom^z(w).
$$
\end{prop}

\begin{proof}
In order to simplify the notation we only consider the case $\lambda>0$. Consider $z \in{\rm Im}(\varphi)$ such that $L^z_\varphi $ is  composed of a unique unbounded connected component.  Let $u$ be a solution of the minimum problem defined by $\psi^z_{T/\lambda}(\lambda w)$; we have $|u(0)|\le \sqrt{m}$, $|u(T/\lambda)-T w|\le  \sqrt{m}$ and
\[
\psi^z_{T/\lambda}(\lambda w) = \frac{\lambda}{T} \int_0^{T/\lambda} |u'(t)|^2 dt = \frac{\lambda}{T} \int_0^T \left|u' \left(\frac{s}{\lambda}\right)\right|^2 \frac{ds}{\lambda}.
\]
Hence, taking $v(s)=u(s/\lambda)$, one has $|v(0)|=|u(0)|\le \sqrt{m}$, $|v(T)-Tw|=|u(T/\lambda)-T w|\le  \sqrt{m}$, so that
\[
\psi^z_{T/\lambda}(\lambda w) = \lambda^2 \frac{1}{T} \int_0^T |v'(s)|^2 ds \ge \lambda^2 \psi^z_T(w).
\]
A similar argument starting from a minimizer of $\psi^z_{T}( w)$ gives the opposite inequality, thus proving the homogeneity of $\psihom^z$ passing to the limit for $T\to+\infty$.

If one takes $z \in{\rm Im}(\varphi)$ such that $L^z_\varphi $ contains no unbounded connected components, proof is trivial, since $ \psi^z_T(w )=\psi^z_T(\lambda w )=+\infty$, for $\lambda \neq 0$ and $T$ large enough.
%
\end{proof}

\begin{rmk}\rm From the previous propositions we obtain that, $\psihom$ is convex, finite and homogeneous of degree two.
\end{rmk}
\goodbreak

\section{Proof of the homogenization result}\label{sec gamma lim}

We subdivide the proof into a lower and an upper bound.

\subsection{Lower bound}
We want to prove that, for any sequence $\ueps$ converging in the strong topology of $L^2([0,1];\Rn{m})$ to a function $u \in L^2([0,1];\Rn{m})$ as $\eps \to 0$, one has
\begin{equation}\label{liminf ineq}
\liminf_{\eps \to 0} \Feps(\ueps) \ge  F(u).
\end{equation}

First of all observe that we can assume, without loss of generality, that
\begin{equation}\label{eq equibdd}
\Feps(\ueps)\le  \lambda < +\infty \quad \hbox{ for all } \eps > 0,
\end{equation}
otherwise the $\Gamma \hbox{-}\liminf$ inequality (\ref{liminf ineq}) is trivial. By the equiboundedness of $\Feps$, we also deduce that $\int_0^1 |\ueps'|^2 dt < \lambda$, so that $\|\ueps \|_{H^1}$ is equibounded.  Hence, we also have $\ueps  \weakconv u$ weakly in $H^1([0,1];\Rn{m})$.

Since $\varphi$ is continuous and periodic, Im$(\varphi)$ is a bounded set. Upon adding a constant vector to $\varphi$ we may suppose that Im$(\varphi) \subseteq [0,b]^{m'}$
for some constant $b>0$. With fixed $N\in{\mathbb N}$ we subdivide $[0,b]^{m'}$ in $N^{m'}$ cubes $Q_j=Q^N_j$ with disjoint interior, edge
of side length $b/N$ and centre $z_j$; i.e.,
\[
Q_j= z_j + \frac{b}{N} \left[ -\frac{1}{2}, \frac{1}{2} \right]^{m'}
\quad  j\in\{1,\dots,N\}^{m'},
\]
so that $Q_j$ and $Q_k$ may intersect only at their boundary, and  $\bigcup_{j}\partial Q_j$ is a part of a cubic lattice in $\Rn{m'}$ of edge $b/N$.

We now use the subdivision of the image of $\varphi$ into cubes to split the domain $[0,1]$ of $\ueps$ into subintervals where $\varphi(u_\e)$ is ``almost constant''. To that end, we
construct an increasing finite sequence $t_i=t_i^{N,\e}$ as follows. We first set $t_0=0$, and correspondingly $z^0$ any centre $z_j$ of a cube such that $\varphi(u_\e(0))\in Q_j$
(which may be not unique if $\varphi(u_\e(0))\in \partial Q_j$).
 We define recursively
\begin{equation}\label{def ti}
t_{i+1}= \min \Bigl\{ t\in (0,1): 
t\ge t_{i}, \varphi_{\eps}(\ueps(t)) \in \partial \Bigl(z^{i} + \frac{b}{N}\Bigl[-\frac{3}{2}, \frac{3}{2} \Bigr]^{m'}\Bigr) \Bigr\},
\end{equation}
if this set is not empty. In this case, we choose as $z^{i+1}$ the centre of any cube $Q_j$ such that $\varphi(u_\e(t_{i+1}))\in \partial Q_j$. 
  Note that $t_{i+1}>t_i$ since $\varphi_{\eps}(\ueps(t_i)) \not\in \partial \Bigl(z^{i} + \frac{b}{N}\bigl[-\frac{3}{2}, \frac{3}{2} \bigr]^{m'}\Bigr)$. If the set in (\ref{def ti}) is empty then we define $t_{i+1}=1$, and stop the iteration procedure.
%
%


With this subdivision of $[0,1]$ we want to highlight the instants $t_i$ where $\varphi_{\eps}(\ueps)$ is passing from a cube $Q_j$ to a neighbouring one $Q_k$.
Note that, for two consecutive $t_i$, by definition, one has
\begin{equation}\label{stinco}
|\varphi_{\eps}(\ueps(t_{i+1}))-\varphi_{\eps}(\ueps(t_{i}))| \ge \frac{b}{N}.
\end{equation}

Also note that, if $\ueps$ always lies in the cube
$z_0 + \frac{b}{N}\bigl(-\frac{3}{2}, \frac{3}{2} \bigr)^{m'}$ then our set of points is made just by $t_0=0$ and $t_1=1$.

In the procedure described above
it is not a priori clear if the number of $t_i$ is finite for $\eps$ and $N$ fixed.
This is made precise by the following result.

\begin{lemma}\label{lemma JepsN}
For any $N > 0$ and $K>0$, there exists $\eps_0$ such that for all $\eps < \eps_0$
the set $\left\{ t_i\right\}_i$ defined in {\rm(\ref{def ti})} is finite.
If we denote $$
J^N_{\eps}:= \sharp (\left\{ t_i \right\}_i)-1m
$$
then we have
\begin{equation}\label{eq >Keps}
\Delta t_i:=t_{i}-t_{i-1} \ge  K \eps^2,\quad i\in\{1,\dots, J^N_{\eps}\}.
\end{equation}
\end{lemma}

\begin{proof}
 For any $i\in\{1,\dots,J^N_{\eps}\}$, using Jensen's inequality, the hypothesis (\ref{eq equibdd}) and the estimate (\ref{stinco}), we have
\begin{eqnarray}\nonumber
\lambda&>&\Feps(\ueps) \ge  \delta^2 \int_{t_{i-1}}^{t_i} |\varphi_{\eps}(\ueps)'|^2 dt
\\ \nonumber
&\ge & \frac{\delta^2}{\Delta t_i} (\varphi_{\eps}(\ueps(t_i)) - \varphi_{\eps}(\ueps(t_{i-1})))^2
 \\
&\ge& \frac{\delta^2}{\Delta t_i} \frac{b^2}{N^2}.\label{stit}
\end{eqnarray}

Since $\delta \to 0$ and $\delta / \eps \to +\infty$ as $\eps \to 0$ there exists $\eps_0>0$ such that, for any $\eps < \eps_0$ one has $\delta \ge  \eps$. Hence we get
\[
\Delta t_i \ge  \frac{\eps^2 b^2}{\lambda N^2},
\]
Summing up in $i$ we then obtain
\begin{equation}\label{stille}
J^N_{\eps}\le  \frac{\lambda N^2}{b^2\eps^2} < +\infty.
\end{equation}

From (\ref{stit}) we also immediately deduce (\ref{eq >Keps}) as
$$
\Delta t_i \ge  {\delta^2 b^2\over\e^2 \lambda N^2}\, \eps^2,
$$
and $\delta/\e\to+\infty$.
\end{proof}
%

We note that equation (\ref{eq >Keps}) gives an upper bound for the number of the intervals of the partition $\{t_i\}$. Since we want to use this partition to define a piecewise-affine approximation of the target function $u$ we need to possibly refine it. To that end we fix $M\in{\mathbb N}$ and introduce a new partition $\overline{t}_i$, for $i=1, \dots, \overline{J}^{N,M}_{\eps}$, subdividing each interval such that
$\Delta t_i\ge {1\over M}$ possibly adding other points to $\{t_i\}$, in such a way that the new partition satisfies
\[
\overline{\Delta t}_i= \overline{t}_i-\overline{t}_{i-1}<\frac{1}{M}.
\]
We also suppose that the new partition is not too fine; i.e., that
\begin{equation}\label{cort}
\overline{\Delta t}_i= \overline{t}_i-\overline{t}_{i-1}>\frac{1}{2M}
\end{equation}
if either $\overline{t}_i$ or $\overline{t}_{i-1}$ do not belong to the original partition.
In particular, each interval is subdivided in at most $2M$  subintervals,
so that, if we denote by $\overline{J}^{N,M}_{\eps}$ the total number of the
intervals of the new partition, by (\ref{stille}),
\[
\overline{J}^{N,M}_{\eps} \le  \frac{\lambda N^2}{\eps^2 b^2} 2M.
\]

Note that by construction we have
\begin{equation}\label{eq phiti}
|\varphi_{\eps}(\ueps(s))- \varphi_{\eps}(\ueps(s'))|<\frac{3b}{N}\sqrt{m'} \quad s,s' \in [\overline{t}_{i-1},\overline{t}_{i}];
\end{equation}
and
\begin{equation}
K \eps^2 < \overline{\Delta t}_i < \frac{1}{M}
\end{equation}
for all $\eps < \eps_0$.

Now we single out a type of intervals that can be neglected in the estimate of
$\Feps(\ueps)$. With fixed $K,M,N$ consider the set of indices
\[
I^K_{\eps}=\left\{i\in\{1,\ldots,\overline{J}^{N,M}_{\eps}\}: |\overline{t}_{i}-\overline{t}_{i-1}|< K\eps \right\}
\]
and the respective set of intervals $B^K_{\eps}=\bigcup_{i \in I^K_{\eps}}[\overline{t}_{i-1}, \overline{t}_{i}] \subset [0,1]$. We also suppose that $\e$ is small enough so that
$$
\e K\le {1\over 2M}.
$$
In such a way, by (\ref{cort}), the endpoints of each such interval both belong to the original partition $\{t_i\}$ and (\ref{stinco}) holds.

Arguing as in the proof of Lemma \ref{lemma JepsN}, we have
\[
\lambda>\Feps(\ueps) \ge  \frac{\delta^2b^2}{N^2 K \eps} \sharp I^K_\eps
\]
so that
\[
\sharp I^K_\eps \le  \frac{\lambda N^2 K \eps}{\delta^2 b^2}.
\]
Therefore one has
\begin{equation}\label{eq IKeps}
|B^K_\eps|\le  \sum_{i \in I^K_\eps} \Delta t_i \le  K\eps \sharp I^K_\eps \le  {\lambda N^2 K^2\over b^2} \frac{\eps^2}{\delta^2} =o(1)
\end{equation}
as $\eps \to 0$.

We will give a lower bound by estimating $\Feps(u_\e)$ only by the contribution of
the unconstrained part for the intervals not belonging to $I^K_\e$
; i.e.,
\begin{eqnarray*}
\Feps(\ueps)& \ge & \sum_{i\notin I^K_{\eps}} \int_{\overline{t}_{i-1}}^{\overline{t}_{i}} |\ueps'|^2 dt
\\
&= & \sum_{i \notin I^K_{\eps}} \eps \int_0^{\frac{\overline{\Delta t}_i}{\eps}} \left| \veps'(s)\right|^2 ds
=\sum_{i \not\in I^K_{\eps}} {\overline{\Delta t}_i} {1\over T^i_{\eps}} \int_0^{T^i_{\eps}} \left| \veps'(s) \right|^2 ds,
\end{eqnarray*}
where  we have used the change of variable $s=(t-\overline{t}_{i-1})/\eps$, defining $$\veps(s)={\ueps(\eps s + \overline{t}_{i-1})\over\eps} - \Bigl[{\ueps(\overline{t}_{i-1})\over\eps}\Bigr],
$$ so that $\veps'(s)=\ueps'(\eps s + \overline{t}_{i-1})$, and we have set
$$
T^i_{\eps}=\frac{\overline{\Delta t}_i}{\eps}.
$$

If we define
$$w_{\eps}^i= \frac{\ueps(\overline{t}_i)-\ueps(\overline{t}_{i-1})}{\overline{\Delta t}_i},
$$
then we have
\[
|\veps(0)|=\left|\frac{\ueps(\overline{t}_{i-1})}{\eps} - \left[\frac{\ueps(\overline{t}_{i-1})}{\eps} \right] \right| \le  \sqrt{m},
\]
\[
\left| \veps(T^i_{\eps}) - T^i_{\eps}w_{\eps}^i \right|=\left| \frac{\ueps(\overline{t}_{i-1})}{\eps} - \left[\frac{\ueps(\overline{t}_{i-1})}{\eps} \right] \right| \le  \sqrt{m},
\]
and, by (\ref{eq phiti}) and the periodicity of $\varphi$
\begin{eqnarray*}
\left|\varphi \left(\veps(s)\right) - \overline z_i\right| &=& \left|\varphi \left( \frac{\ueps(s\eps +\overline{t}_{i-1})}{\eps}-\left[ \frac{\ueps(\overline{t}_{i-1})}{\eps} \right]\right) - \overline z_i\right|
\\
&=&\left| \varphi_{\eps}(\ueps(s\eps +\overline{t}_{i-1})) -\overline z_i\right| \le  \frac{\alpha}{N}, \quad \hbox{ if } s \in [0,T^i_{\eps}],
\end{eqnarray*}
where $\overline z_i$ is one of the centres $\{z^i\}$ of cubes $Q_j$  given by the construction of the partition $\{t_i\}$, corresponding to the interval $[\overline{t}_{i-1}, \overline{t}_i]$ and $\alpha$ depends on $b$ and $m'$.
Then, using definition (\ref{def psiTzc}), one has
\[
\Feps(\ueps) \ge  \sum_{i \notin I^K_{\eps}} \overline{\Delta t}_i \psi_{T^i_{\eps}}^{\overline z^i,\frac{\alpha}{N}} (w_{\eps}^i).
\]

Now we use the hypothesis in Definition \ref{hp geom}.
For any $i \notin I^K_{\eps}$ and a fixed $N>0$, there exists $z'_i$ and $\overline{w}^i_{\eps}$ such that
\[
\psi_{T^i_{\eps}}^{\overline z_i,\frac{\alpha}{N}} (w_{\eps}^i) \ge  \left(1-\omega\left(\frac{\alpha}{N} \right) \right) \psi_{T^i_{\eps}}^{z'_i}(\overline{w}^i_{\eps}) - \frac{k(N)}{T^i_{\eps}}
\]
and
\begin{equation}\label{eq ww'}
|w^i_{\eps}-\overline{w}^i_{\eps}|\le  \frac{\sqrt{m}}{T^i_{\eps}}.
\end{equation}

Note that $T^i_{\eps}\ge K$, so that, upon supposing $K$ large enough, the finiteness of $\psi_{T^i_{\eps}}^{z'_i}(\overline{w}^i_{\eps})$ implies that either $ \overline{w}^i_{\eps}=O(1/T^i_{\eps})$ or all sets $\{\varphi= z'_i\}$
have a unique unbounded connected component. In the first case, since, by Proposition \ref{prop psihom 2om} $\psihom(w) \le c |w|^2$, we can estimate
$$
\psi_{T^i_{\eps}}^{\overline z^i,\frac{\alpha}{N}} (w_{\eps}^i)\ge 0\ge
\psihom (w_{\eps}^i)- C{1\over (T^i_{\eps})^2},
$$
while in the second case we can use (\ref{stisti}) in Remark \ref{unist}, and the estimate $\psihom^{z'_i}\ge\psihom$.
Summing up (and recalling that  $T^i_{\eps}\ge K$) we get
\begin{eqnarray}\nonumber
\Feps({\ueps}) &\ge &
 \sum_{i \notin I^K_{\eps}} \overline{\Delta t}_i \left( \left(1-\omega\left(\frac{\alpha}{N} \right) \right) \psi_{T^i_{\eps}}^{z_i'} (\overline{w}_{\eps}^{i}) - \frac{k(N)}{T^i_{\eps}} \right)
\\
\nonumber
&\ge & \left(1-\omega\left(\frac{\alpha}{N} \right) \right)
\left(1-\frac{1}{K+1}\right)
 \sum_{i \notin I^K_{\eps}} \overline{\Delta t}_i \Bigl( \psihom (\overline{w}_{\eps}^{i}) - \frac{k'(N)}{K}\Bigr).
\\
&\ &\ \label{tret}
\end{eqnarray}

If we consider the piecewise-affine functions $u^{N,M,K}_\e$ defined by
$u^{N,M,K}_\e(0)= u_\e(0)$ and
$$
{d\over dt} u^{N,M,K}_\e(t)=\begin{cases}\overline{w}_{\eps}^{i} &\hbox{if $t\in [\overline t_{i-1},
\overline t_{i}]$, $i\not\in I^K_\e$}\cr
0 &\hbox{ otherwise}\end{cases}
$$
almost everywhere in $[0,1]$, then we have
$$
\sum_{i \notin I^K_{\eps}} \overline{\Delta t}_i \psihom (\overline{w}_{\eps}^{i})
=\int_0^1 \psihom\Bigl({d\over dt} u^{N,M,K}_\e\Bigr)dt.
$$
Moreover, if we denote by $u_{N,M,K}$ any of the limits of $ u^{N,M,K}_\e$ in $L^2([0,1];\Rn{m})$
(which exist up to subsequences), then, by the lower semicontinuity of $v\mapsto \int_0^1 \psihom(v')dt$, we have
$$
\liminf_{\e\to0} F_\e(u_\e)\ge \left(1-\omega\left(\frac{\alpha}{N} \right) \right)
\left(1-\frac{1}{K+1}\right)\int_0^1\psihom(u_{N,M,K}')dt- \frac{k'(N)}{K}.
$$
By (\ref{eq IKeps}) and (\ref{eq ww'}) we have
$$
\|u_{N,M,K}-u\|_{L^2}= o(1)
$$
as $K\to+\infty$, $N\to+\infty$ and eventually $M\to+\infty$.
Again by lower semicontinuity we finally obtain
$$
\liminf_{\e\to0} F_\e(u_\e)\ge \int_0^1 \psihom(u')dt
$$
as desired.

\subsection{Upper bound}
Given $u \in H^1([0,1];\Rn{m})$, we want to find a sequence $\ueps \in H^1([0,1];\Rn{m})$ such that
\begin{equation}\label{limsup eq}
\limsup_{\eps \to 0} \Feps(\ueps) \le  F(u).
\end{equation}
Since the integrand $\psihom$ is a convex, lower-semicontinuous and coercive function,
we can follow a standard approximation procedure, and prove (\ref{limsup eq}) only when $u$ is a linear (or affine) function. The general case is obtained by localizing the construction for piecewise-affine functions and in general by the strong density of such functions in $H^1$.

\smallskip
We consider the case $u(t)= wt$ with $w\in\Rn{m}$.
By Caratheodory's theorem there exist vectors $w_i \in \Rn{m}$ and coefficients $\lambda_i \in [0,1]$, for $i\in\{1, \dots, m+1\}$, such that

\begin{equation}\label{carath eq1}
\sum_{i=1}^{m+1} \lambda_i =1, \quad \quad \sum_{i=1}^{m+1} \lambda_i w_i = w,
\end{equation}
and
\begin{equation}\label{carath eq2}
\psihom(w)=\sum_{i=1}^{m+1}\lambda_i \psihom^{z_i}(w_i).
\end{equation}


For any fixed $T>0$ and $i\in\{1, \dots, m+1\}$ let $u_i:[0,\lambda_i T] \to \Rn{m}$ be a
solution for the minimum problem
\begin{eqnarray*}
\psi_{\lambda_i T}^{z_i}(w_i)&= &\frac{1}{\lambda_i T} \min \Bigl\{ \int_0^{\lambda_i T} |v'|^2 dt \quad |v(0)-0|\le\sqrt{m},
\\
&&\qquad\qquad |v(\lambda_i T)-\lambda_i T w_i|\le\sqrt{m}, \varphi_{\eps}(v)=z_i \Bigr\}.
\end{eqnarray*}

We want to construct a recovery sequence by a patchwork procedure using such $u_i$. To that end, we fix a positive constant $K$ and we define (for notational convenience we set $\lambda_0=0$)
\begin{eqnarray*}
a_j=\sum_{i=0}^{j-1} \lambda_i T + (j-1)K,
\quad
b_j=\sum_{i=1}^{j} \lambda_j T + (j-1)K,
\quad
c_j&=&\sum_{i=1}^j \lambda_i T + jK.
\end{eqnarray*}
for $j\in\{1,\dots,m+1\}$.
We also denote by $\gamma_j:[0,K] \to \Rn{m}$ the parameterization with constant velocity
of the line segments from ${\sum_{i=0}^{j-1} [\lambda_iTw_i]+}u_j(\lambda_j T)$ to $u_{j+1}(0) + \sum_{i=1}^j [\lambda_iTw_i]$, for $j=1,\dots,m+1$ (where we  have also set $w_0=0$ and $u_{m+2}=u_1$ for notational convenience). Note that, by the boundary conditions on $u_j$, one has $l(\gamma_j)\le  2\sqrt{m} $, so that
\begin{equation}\label{eq length lambdaj}
|\gamma_j'| \le  \frac{2\sqrt{m}}{K}.
\end{equation}

We construct the function $\tilde{u}: [0,\sum_{i=1}^{m+1} \lambda_i T + (m+1)K]=[0,T+(m+1)K] \to \Rn{m}$ as
\begin{equation}\label{eq def u tilde}
\tilde{u}(t)=\begin{cases}\displaystyle{\sum_{i=0}^{j-1} [\lambda_iTw_i]+}
		u_j\bigl(t - a_j\bigr)	&	a_j \le  t  \le b_j\\
															
		\gamma_j\left(t - b_j\right)	& b_j\le  t   \le c_j,
																									
		\end{cases}
\end{equation}
on the segment $[a_j,c_j]$ for $j\in\{1,\dots,m+1\}$.

Then we define $u_{TK}:[0,T] \to \Rn{m}$ as
\begin{equation}\label{eq def u}
u_{TK}(t)=\tilde{u}\left( \frac{t}{T} (T+(m+1)K) \right).
\end{equation}
Note that $u_{TK}(0)=u_1(0)$ and $u_{TK}(T)=u_1(0)+\sum_{i=1}^{m+1}[\lambda_iTw_i]$. This implies that the function
$$
u_{TK}(t)- {1\over T}\sum_{i=1}^{m+1}[\lambda_iTw_i] t
$$
can be extended $T$-periodically. We still denote by $u_{TK}$ the function resulting from this extension.
If we choose $T=T_\e\to+\infty$ and $K=K_\e\to+\infty$ such that  $K<\!<T$ then
 we have
\[
G(K,T):=
\frac{1}{T}(T+(m+1)K) \xrightarrow[K \to \infty]{} 1.
\]
so that, since ${1\over T}\sum_{i=1}^{m+1}[\lambda_iTw_i] \to w$,
the functions
$$
u_\e(t)= \e\,u_{TK}\Bigl({t\over\e}\Bigr)
$$
converge to the function $u(t)=wt$.

By the periodicity of $u_{TK}$, using (\ref{eq length lambdaj}), we can estimate
\begin{eqnarray*}
\limsup_{\e\to0} F_\e(u_\e)&=&
\limsup_{\e\to0} {1\over T+(m+1)K}\Biggl(\sum_{j=1}^{m+1} \lambda_i T \psi_{\lambda_iT}^{z_i}(w_i)\\
&& +\sum_{j=1}^{m+1} \Bigl(\int_0^K |\gamma'_j|^2dt
+{\delta^2\over\e^2}\int_0^K\Bigl|\nabla_y\varphi\Bigl({\gamma_j\over\e}\Bigr)\gamma'\Bigr|^2dt\Bigr)\Biggr)\\
&\le& \limsup_{\e\to0}
{1\over T}\Biggl(T\psihom(w)+ {C\over K}
+{\delta^2\over\e^2}\|\nabla\varphi\|^2_\infty{C\over K}\Biggr)
\\
&=& \psihom(w)+\|\nabla\varphi\|^2_\infty C\limsup_{\e\to0}
{\delta^2\over\e^2}{1\over TK}.
\end{eqnarray*}
If we choose, for example, $T=\delta^{1/3}/\eps$ and $K=\delta^{2/3}/\eps$, the conditions $T>\!> K>\!> 1$ are satisfied, and the last term in this estimate vanishes, as desired.

\begin{rmk}[multi-scale nature of recovery sequences]\rm
It is interesting to note that the construction of the recovery sequence departs from the
usual scaling of a fixed periodic function in the use of two scales $T$ and $K$, the first
necessary as customary to use the correctors given by the homogenization formula, and the second one, slower, but still tending to infinity, to construct the junctions between the functions defined by the correctors.
\end{rmk}

\section{Examples}\label{esempi}

In this section we include some examples in two and three dimension, that in particular show that
the limit may not be a quadratic energy. 

\begin{example} \label{sinx siny}\rm
We consider the constraint function: $\varphi:\Rn{2} \to \mathbb{R}$
\[
\varphi(x,y)= \sin (2\pi x) \sin(2 \pi y).
\]
We want to show that the homogenized function of the oscillating constrained problem associated to $\varphi$ is the squared $l_1$ norm: $\psi: \Rn{2} \to [0,+\infty)$,
 \[
 \psihom(w)=(|w_1|+|w_2|)^2
 \]

First of all note that $\varphi$ has {\em non-degenerate} level sets, in the sense of Definition~\ref{HyH}: we either have the level $\left\{\varphi=0 \right\}$, that has only one unbounded connected component, the lattice with vertices in $\mathbb{Z}^2$, or the level $\left\{\varphi= c\right\}$, with $c \neq 0$, that is composed of infinitely many bounded connected components. Moreover the set $\left\{\varphi=0 \right\}$ is made by a union of $\mathcal{C}^1$ sets, so that it can be proved that it satisfies the property in Definition \ref{esistenza raccordi}.  The
property in Definition \ref{hp geom} is proved by noting that to a curve $u$ lying in
$\{|\varphi|\le c\}$ we can associate a curve $u_0$ composed by its projection on $\{\varphi=0\}$ where this projection is uniquely defined, and piecewise-linear joints still lying in $\{\varphi=0\}$ on neighbourhoods of $\mathbb{Z}^2$ of radius $c$. The Dirichlet integral of $u_0$ is then not greater than $1+O{(}c)$ that of $u$, up to a small error due to the endpoints of $u$.

We apply Theorem \ref{thm gammalim} getting the homogenization formula (\ref{eq psihom}), that, in this particular case, is simpler: being $\left\{\varphi=0 \right\}$ the only level set with an unbounded connected component, we have
\[
\psihom(w)=\psihom^0(w).
\]

Now we want to prove that
\begin{equation}\label{upper}
\psihom^0(w)\le  (|w_1|+|w_2|)^2=:|w |^2_1.
\end{equation}

Given a general curve $u:[0,T] \to \Rn{2}$ satisfying the conditions $|u(0)|\le\sqrt{2}, \quad |u(T)-Tw|\le\sqrt{2}, \quad \varphi(u)=0$, using the change of variable $s=t/T$, the function $v(s)=u(sT)$ and Jensen's inequality, one has
\[
\int_0^T |u'|^2dt = \int_0^1|u'(sT)|^2 Tds = \frac{1}{T} \int_0^1 |v'(s)|^2 ds \ge  \frac{1}{T} \left( \int_0^1 |v'| ds \right)^2.
\]

Note that the last term depends on the length of the curve $v$, lying in the lattice $\varphi=0$,
and it is larger than or equal to $|T\tilde{w} |^2_{1}$, where
$$\tilde{w}={1\over T}(u(T)-u(0)).$$
Therefore, since $ |T\tilde{w}| \ge  |Tw|-2\sqrt{2}$,
\[
\int_0^T |u'|^2dt \ge  \frac{1}{T} |T\tilde{w} |^2_{1} \ge  \frac{1}{T} (|Tw |_{1}-2\sqrt{2})^2.
\]
Now we take the infimum on curves $u$ of this type and then the limit for $T \to \infty$, getting
\[
 \psihom^0(w) \ge  \lim_{T \to \infty} \frac{1}{T^2} (|Tw |_{1}-2\sqrt{2})^2=| w|_{1}^2.
\]
Since the value $|w|_{1}^2$ is asymptotically achieved on all test curves which are simple and have constant velocity, the equality is proved.
\end{example}

\begin{example}\label{esempio palle 2d}
\rm
In this second example we consider the oscillating constraint defined by $\varphi(x,y)= $ dist$((x,y), \mathbb{Z}^2)$; i.e., the distance from the points with integer coordinates. The corresponding homogenized function is the following norm, defined on the whole $\Rn{2}$:
\[
\psi(w)=\left(| w |_{\infty} \frac{\pi}{2}\right)^2=\left(\max\left\{ w_1,w_2\right\} \frac{\pi}{2}\right)^2.
\]

As in Example \ref{sinx siny}, the constraint $\varphi$ satisfies the properties in Definitions  \ref{hp geom}, \ref{esistenza raccordi} and \ref{HyH}, so that the homogenization formula (\ref{eq psihom}) can be used. Again the only level set with an unbounded connected component is the set $\{\varphi=1/2 \}$, hence
\[
\psihom(w)=\psihom^{1/2}(w).
\]

We can suppose without loss of generality that out test curves
have endpoints with one of the two coordinates an integer.
In this case note that  if
$u$ satisfies the boundary conditions of $\psihom^{1/2}$:
\[
|u(0)|\le\sqrt{2}, \quad |u(T)-Tw|\le\sqrt{2}, \quad \varphi(u)=1/2,
\]
then we have
\begin{equation}\label{prop l min}
\int_0^T |u'| dt \ge  |u(T)-u(0)|_{\infty} \frac{\pi}{2}
\end{equation}
as pictured in Fig.~\ref{esempio3b}.

\begin{figure}[!h]\centering
\includegraphics[width=0.5\columnwidth]{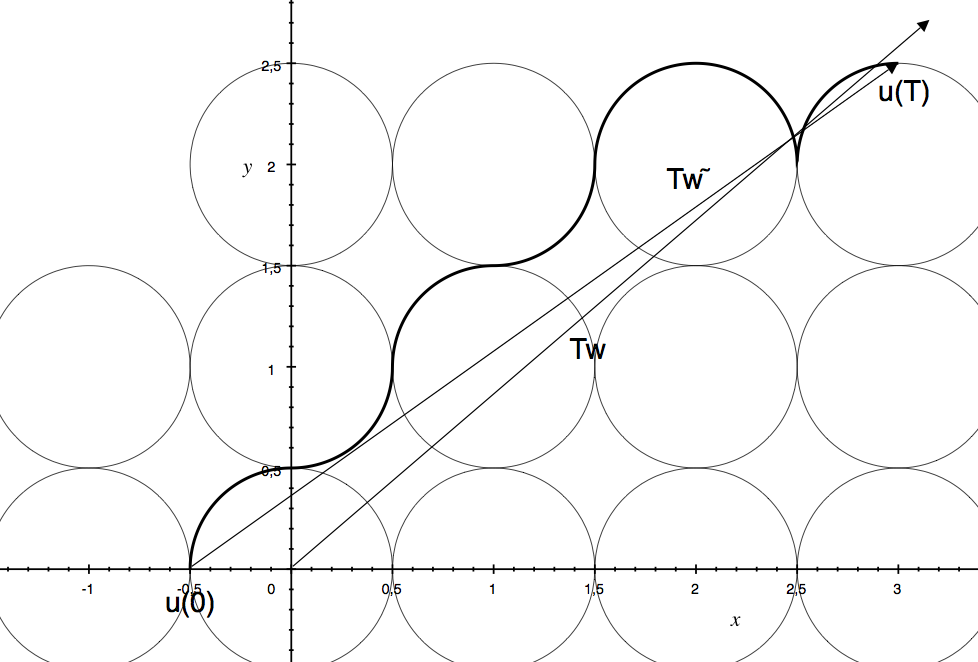}
\caption{An example of a constrained curve $u$ with minimum length.} \label{esempio3b}
\end{figure}

Take a general curve satisfying the conditions defining $\psi_T^{1/2}$: then, using change of variable $s=t/T$ and Jensen's inequality, with the same notation as in the previous example, we have
\[
\int_0^T |u'|^2 dt = \frac{1}{T} \int_0^1 |v'|^2 ds \ge  \frac{1}{T} \left( \int_0^1 |v'| ds\right)^2.
\]
Now by  (\ref{prop l min}), we know that the length of $v$ is larger than $|Tw|_{\infty} \pi/2$, so that
\[
\int_0^T |u'|^2 dt  \ge  \frac{1}{T} \left( |Tw|_{\infty} \frac{\pi}{2} \right)^2.
\]
This inequality holds for any curve $u$ satisfying the boundary conditions for $\psi_T^{1/2}(w)$ and for any $T>0$, so we can take the infimum over $u$ and the limit as $T \to \infty$, getting
\[
\psihom(w) \ge  \lim_{T \to \infty}  \frac{1}{T^2} \left( |Tw|_{\infty} \frac{\pi}{2} \right)^2 = \left( |w|_{\infty} \frac{\pi}{2} \right)^2
\]
The optimality of this estimate is shown by computing the energy of test curves as those pictured in
Fig.~\ref{esempio3b} with constant velocity.
%

%

%








\end{example}

\begin{example}\label{esempio 3d faces}
\rm
We consider an example for curves in $\Rn{3}$. In order to choose the constraint function $\varphi$ we can introduce the network
\[
\mathcal{L}=\{(x,y,z) \in \Rn{3}: x\in\mathbb{Z} \hbox{ or }  y\in\mathbb{Z} \hbox{ or } z\in  \mathbb{Z}\},
\]
shown in Figure \ref{fig:reticolo3d}, and define $\varphi(x,y,z)=$ dist$^2((x,y,z),\mathcal{L})$, or $$\varphi(x,y,z)=\min\{\hbox{dist}^2(x,\mathbb{Z}^2),
\hbox{dist}^2(y,\mathbb{Z}^2), \hbox{dist}^2(z,\mathbb{Z}^2)\}.$$

\begin{figure}[!h] \centering
\includegraphics[width=0.5\columnwidth]{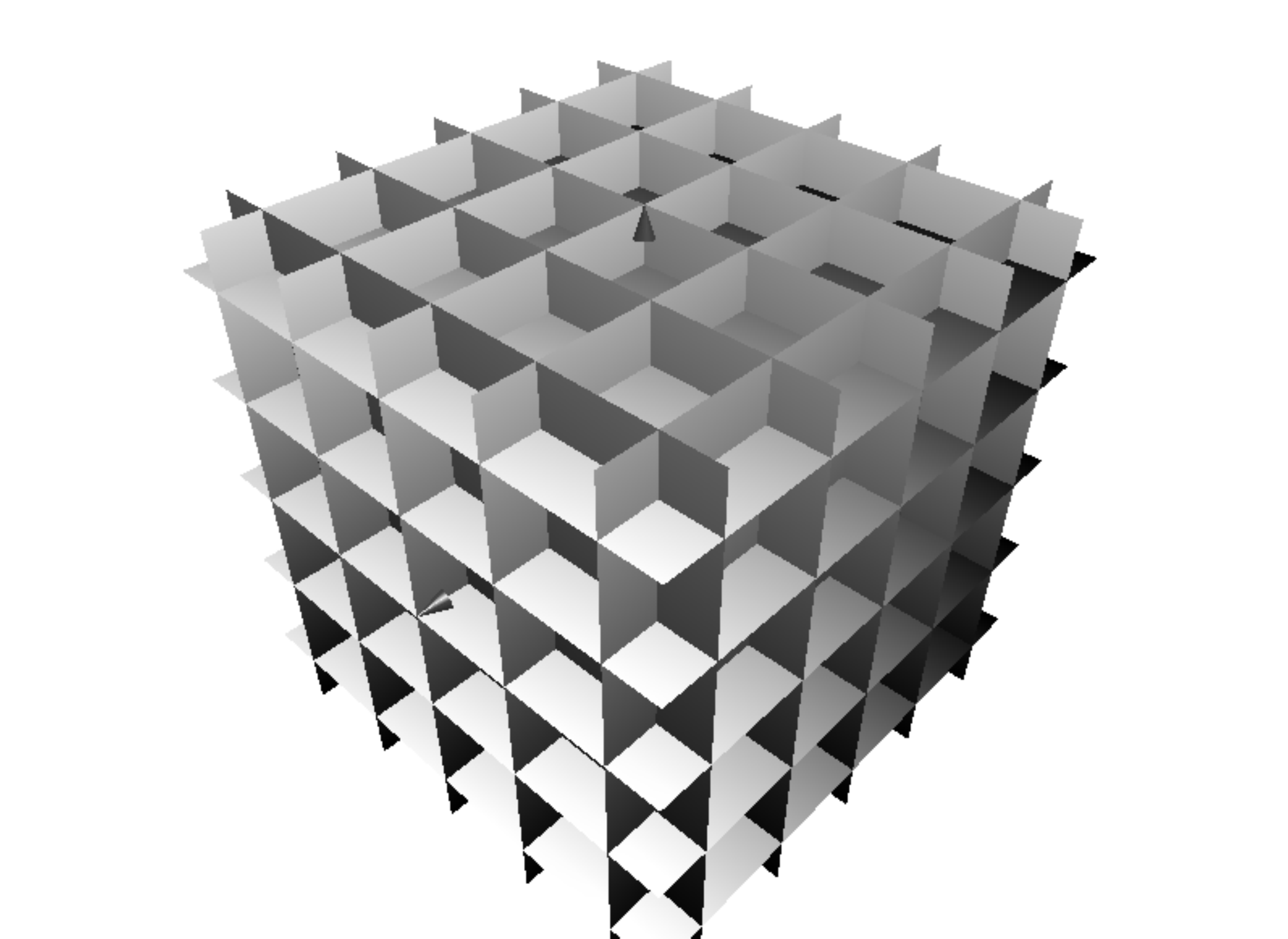}
\caption{The network $\mathcal{L}$.}\label{fig:reticolo3d}
\end{figure}

Actually, $\varphi$ satisfies the condition in Definition \ref{HyH} of non-degenerate levels, as there exists only one unbounded and connected level set of $\varphi$; i.e., the network $\mathcal{L}=\{\varphi=0 \}$. It also satisfies the hypothesis in Definition \ref{esistenza raccordi}, being a union of $\mathcal{C}^1$ sets. Hypothesis of Definition \ref{hp geom} is more tricky to verify: if we take $z \neq 0$ and $c<|z|$ then $\psi^{z,c}_T(w)=+\infty$; moreover it can be proved that for $c$ sufficiently small we have
\[
\psi_T^{0,c}(w) \ge  (1+o_c(1))\psi^0_T(w') - \frac{k(c)}{T},
\]
for any $w,w' \in \Rn{3}$, with $w'=w+o_T(1)$.

Therefore, by Theorem \ref{thm gammalim} we have
\[
\psihom(w) = \lim_{T \to \infty} \psi_T^0(w),
\]
where $\psi_T^0(w)$ measures the minimal length of a curve from $0$ to $Tw$, lying in the level set $\{\varphi=0 \}$. In order to compute such a metric we can consider the parallelepiped with edges $x=[w_1], y=[w_2], z=[w_3]$, so that its faces belong to the network $\mathcal{L}$. Note that the minimal curve joining $0$ and $(x,y,z)=([w_1],[w_2],[w_3])$ lying first in the plane $y=0$ and then in $x=[w_1]$ has length equal to
\[
\min_{0\le  t \le  z} f(t),\quad f(t)=\left(\sqrt{x^2 + t^2} + \sqrt{y^2+(z-t)^2} \right).
\]

The minimum of $f(t)$ is reached for $t=zx/(y+z)$ and it is
\[
l(x,y,z)=\sqrt{(|x|+|y|)^2+z^2}.
\]

Now we have to find the minimum of $l(x,y,z)$ on the permutations of $x,y,z$: observe that
\[
|y|\le  |z| \Leftrightarrow (|x|+|y|)^2+z^2 \le  (|x|+|z|)^2 + y^2
\]
\[
|x|\le  |y| \Leftrightarrow (|z|+|y|)^2+x^2 \le  (|x|+|y|)^2 + z^2,
\]
so that we have
\[
\psihom(w) = (\min\{|x|,|y|,|z| \})^2 + (|x|+|y|+|z|-\min\{|x|,|y|,|z| \})^2;
\]
that is, the euclidean norm for the minimal component of $w$ added to the $l_1$ norm of the other two components squared.
\end{example}

\begin{example}\label{esempio 3d balls}
\rm
As a second example of curves in $\Rn{3}$ let us consider the following constraint function
\[
\varphi(x,y,z)=\hbox{dist}^2((x,y,z),\mathbb{Z}^3).
\]
We observe that this is the natural generalization to $\Rn{3}$ of Example \ref{esempio palle 2d}. Hence, the {\em non-degenerate} level set $\mathcal{L}=\left\{ \varphi=1/4 \right\}$, pictured in Fig.~\ref{fig:reticolopalle3d}, has a unique unbounded connected component, that satisfies the hypotheses of Definitions \ref{esistenza raccordi} and \ref{hp geom}.

\begin{figure}[!h]\centering
\includegraphics[width=0.7\columnwidth]{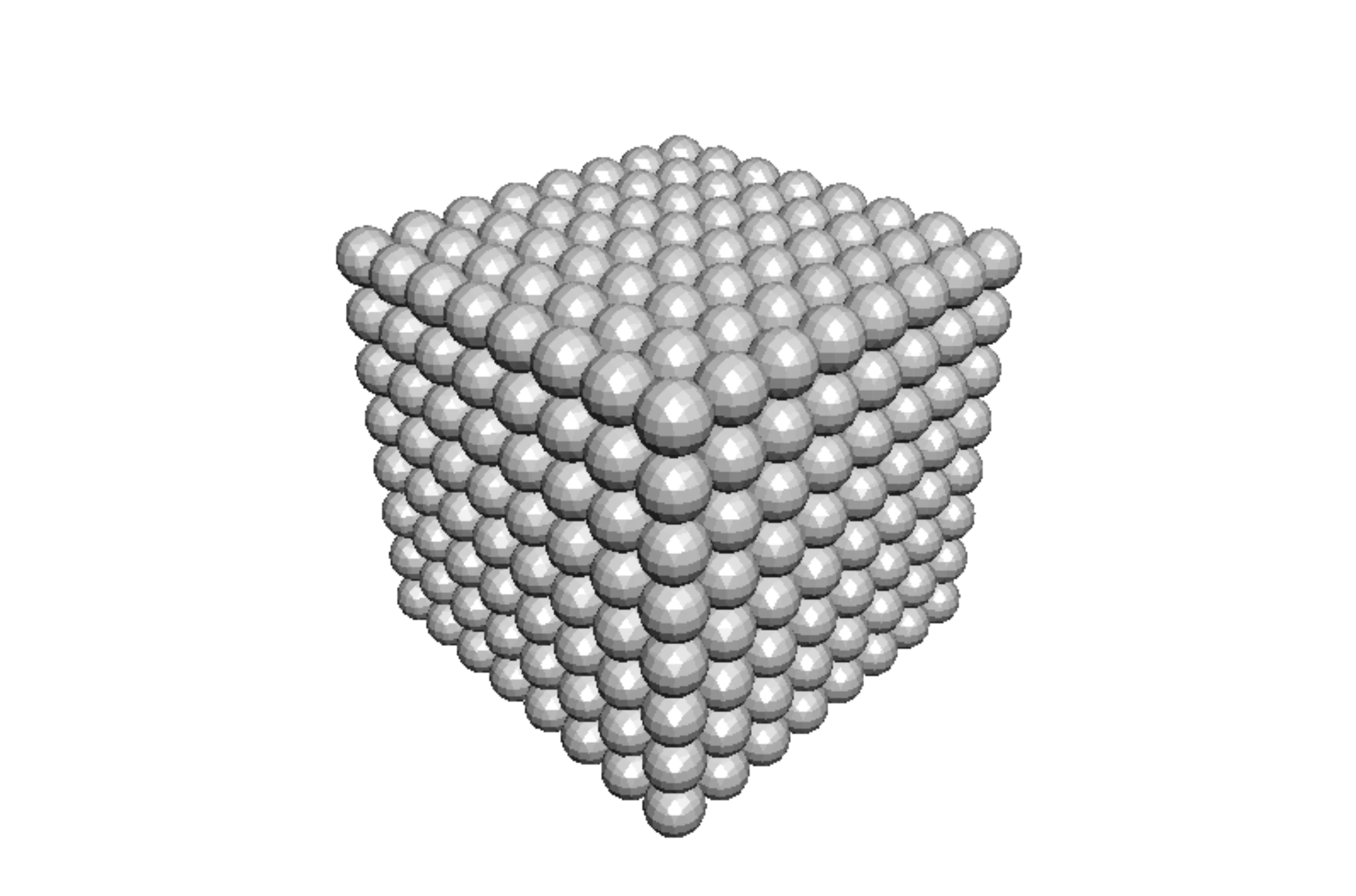}
\caption{The network $\mathcal{L}\subseteq \Rn{3}$.}\label{fig:reticolopalle3d}
\end{figure}

Note that $\mathcal{L}$ is not the only level set having that property; hence we are only able to find an upper bound for $\psihom(w)$. To this end we can use an argument similar to that of Example \ref{esempio 3d faces}. Consider the parallelepiped of edges $x=[w_1]$, $y=[w_2]$ and $z=[w_3]$. The length of a curve $u$, connecting $0$ and $w$, with the strict constraint $\varphi(u)=1/4$, is less or equal to that of a curve $v$, from $0$ to a point $([w_1],0,t)$, with $0 \le  t \le  [w_3]$, lying in the plane $y=0$, and from $([w_1],0,t)$ to $[w]$, in the plane $x=[w_1]$. We can exploit the result of Example \ref{esempio palle 2d} in these two planes, so that we have the minimal length
\[
l(v) = \min_{0 \le  t \le  z} \frac{\pi}{2}\left( |(x,t) |_\infty + |((z-t),y) |_\infty \right),
\]
where $|(\xi_1,\xi_2) |_\infty=\max(|\xi_1|,|\xi_2|)$ is the $l^\infty$norm of the vector $\xi \in \Rn{2}$.

Note that if $z < x+y$ then $|(x,t)|_\infty + |((z-t),y)|_\infty=|x|+|y|$, while if $z \ge  x+y$ then $|(x,t)|_\infty + |((z-t),y)|_\infty=|z|$. Therefore
\[
l(v) = \frac{\pi}{2} \min \left(\max\left( |x|+|y|, |z| \right), \max\left( |x|+|z|, |y| \right), \max\left( |z|+|y|, |x| \right) \right),
\]
%
which, after examining separately the cases $|x|\le  |y|$ and $|y|\le |z|$, can be written as follows:
\begin{eqnarray*}
l(v)&=& \frac{\pi}{2} \max\Bigl\{ \min(|x|,|y|,|z|),|(x,y,z)|_1- \min(|x|,|y|,|z|)  \Bigr\}\\
&=& \frac{\pi}{2} |\left( \min(|x|,|y|,|z|), |(x,y,z) |_1- \min(|x|,|y|,|z|) \right) |_{\infty}.
\end{eqnarray*}
Hence the upper bound for $\psihom$ reads
\begin{eqnarray*}
\psihom(w)& \le & \psihom^{1/4}(w)
\\
&\le & \left( \frac{\pi}{2} |\left( \min(|x|,|y|,|z|), |(x,y,z) |_1- \min(|x|,|y|,|z|) \right) |_{\infty} \right)^2.
\end{eqnarray*}

\end{example}

\section{An application: density of oscillating-constraint problems in Finsler metrics}\label{Chapter Finsler metrics}

By a symmetric Finsler metric in $\Rn{2}$, controlled from below by the Euclidean norm, we mean a function $\psi: \Rn{2} \to [0,+\infty]$ such that
\begin{itemize}
\item[i)] $\psi$ is $2$-homogeneous: $\psi(\lambda w)=\lambda^2 \psi(w)$ for all $w \in \Rn{2}$ and $\lambda \in \mathbb{R}$;
\item[ii)] $\psi$ is convex;
\item[iii)] $\psi(w)\ge  |w|^2$ for all $w \in \Rn{2}$;
\end{itemize}
Observe that from (i) one has $\psi (w)=\psi(-w)$ for all $w \in \Rn{2}$.

If $\psihom$ is an energy density derived from oscillating constraints as above, then it satisfies these conditions, i.e.,
the $\Gamma$-limit of an oscillating constraint problem, for curves with values in  $\Rn{m}$, is a symmetric Finsler metric. In this section we characterize metrics defined by an oscillating surface on $\Rn{3}$; i.e., we consider constraints given by  functions $\varphi: \Rn{2} \to \Rn{}$: more precisely we show that they are dense in Finsler metrics controlled from below by the Euclidean norm, with respect to $\Gamma$-convergence.

This result is close in spirit to that of \cite{B-B-F}, where it is proved that the closure of the metrics obtained by homogenization of Riemannian ones are all Finsler metrics.
Observe that, differently from the case treated in \cite{B-B-F}, we do not require the boundedness of $\psi$ from above; this allows us to treat cases of metrics whose domain is not the whole $\Rn{2}$.

By the hypothesis on $\psi$, we know that its domain, i.e. the set where $\psi$ is finite, has to be a convex cone in $\Rn{2}$, symmetric with respect to the origin and centered at $(0,0)$, hence, since $\varphi$ is convex, it is a subspace of $\Rn{2}$. So, if $\hbox{dom}(\psi) \neq \left\{0 \right\}$, we might have two different cases:

\begin{enumerate}
\item $\text{dom}\,\psi$ is a line through the origin; i.e., a subspace of dimension one ($\mathbb{R})$, so that $\psi$ is finite only in one direction and we have
\[
\sup_{|w|=1} \psi (w) = +\infty;
\]
\item $\text{dom}\,\psi$ is the whole $\Rn{2}$, so that we have
\[
\max_{|w|=1} \psi (w) = M < +\infty.
\]
\end{enumerate}
In the following functions $\varphi$ of type $1.$ and $2.$ will be called {\em degenerate} and {\em non-degenerate} Finsler metrics, respectively.

It is clear that this distinction cannot be extended to the situation of metrics defined on $\Rn{n}$, $n>2$, that, in general, will contain more cases.

In both cases, we want to prove that for any $\eta > 0$ and $\psi$ satisfying conditions (i)--(iii) there exists a periodic function $\varphi=\varphi_{\eta}: \Rn{2} \to \mathbb{R}$, defining the oscillating constraint and the corresponding functionals $F_{\eps}$, such that the homogenized function $\psi_{\eta}$ of the $\Gamma$-limit
\[
\gammalim{F_{\eps} (u)}= \int_0^1 \psi_{\eta}(u') dt
\]
satisfies the inequality
\begin{equation}\label{eq psi eta}
|\psi_{\eta}(w)-\psi(w)| \le  \eta |w|^2
\end{equation}
for all $w$ such that $\psi(w)<+\infty$ and $\psi_\eta(w)=+\infty$ otherwise.

As in Section \ref{sec intro}, we consider the functional $F_{\eps}$ in the unconstrained form, defined for curves with values in $\Rn{2}$: $\Feps: L^2([0,1];\Rn{2}) \to [0,+\infty]$

\begin{equation}\label{def Feps unconstrained Finsler}
F_{\eps}(u)
= \int_{\Omega}\Bigl( |u_{\eps}'|^2 + \Bigl(\frac{\delta}{\eps}\Bigr)^2 \left| \nabla_y \varphi (u_{\eps}) \ueps' \right|^2\Bigr) dx
\end{equation}
where $y\in\Rn{2}$ denotes the variable of $\varphi$.

\subsection{Degenerate Finsler metrics}\label{sup psi infinito}

We consider the case when the domain of the target metric $\psi$ is a vector space of dimension $1$. Note that if we required the weaker approximation condition
$$
\lim_{\eta\to0}\psi_\eta(w)=\psi(w)
$$
instead of (\ref{eq psi eta}) then this case could be seen as a limit of non-degenerate metrics. We present a construction which allows to directly obtain exactly $\psi$ as homogenized energy density.

It is not restrictive to assume that dom$(\psi)=\left\{ (w_1,w_2) \in \Rn{2}: w_1=0\right\}$ ,as all the other cases can be obtained from this one simply by a change of basis of $\Rn{2}$. Since the only $2$-homogeneous function in one variable is quadratic, then there exists a constant $k>0$ such that
$$\psi(w)=\begin{cases}k |w|^2 &\hbox{ if $w_1=0$}\cr
+\infty &\hbox{ if $w_1\neq0$}.
\end{cases}
$$

Now we construct a periodic function $\varphi$, defining the oscillating constraint, such that the density function of the $\Gamma$-limit of this problem is the quadratic function $\psi$. To that end, consider any $1$-periodic smooth function $g: \mathbb{R} \to \mathbb{R}$ and take $\varphi: \Rn{2} \to \mathbb{R}$, with $\varphi(x,y)=\sin(2 \pi(x-g(y)))$. The level sets $L^z_\varphi $ are defined by the equation $x-g(y)=c$, for a suitable $c \in \Rn{}$. We can represent them as the graph of the function $x=g(y)+c$, so that they are all the same graph up to a horizontal translation.

Note that all level sets of $\varphi$ are made by infinitely many disjoint unbounded  connected component, so that it does not satisfy the {\em non-degenerate} levels condition in Definition \ref{HyH}. However, in this special geometry Theorem \ref{thm gammalim} can be proved without that hypothesis, with
\[
\psihom(w)= \lim_{T\to+\infty} \psi_T^0(w)
\]
(where the value $0$ can be substituted by any vale in $[-1,1]=$ Im$(\varphi)$). In particular
the existence of the limit in this formula can be obtained
as in Lemma \ref{lemma ex limit} following the usual subaddditive argument, which is actually easier to implement in this case.
%

Given $k$, we have to find a function  $g$ such that $\psi_{\hom}=\psi$. To this end, it is enough to impose the identity for all $w$ with $|w|=1$, i.e. that
$\psi_{\hom}(w)=k$. Given $w$  with $|w|=1$ we have:
\begin{eqnarray*}
\psi_{\hom}(w)&=& \lim_{T \to \infty} \frac{1}{T} \min \Bigr\{ \int_0^T |u'|^2dt:  \varphi(u)=0,\\ &&\qquad |u(0)|\le\sqrt{2},|u(T)-Tw|\le\sqrt{2} \Bigl\}.
\end{eqnarray*}
By convexity, this minimum is achieved if $|u'|$ is constant; taking into account the form of the constraint we have
\[
\psi_{\hom}(w) = \left( \int_0^1 \sqrt{1+g'(s)} ds \right)^2.
\]
This shows that $g$ must satisfy
\[
\left( \int_0^1 \sqrt{1+g'(s)} ds \right)^2=k.
\]

\subsection{Non-degenerate Finsler metrics}
In this section we assume that $\psi$ is a Finsler metric satisfying dom$\,\psi=\Rn{2}$. Given $\eta>0$,
we want to construct a function $\varphi$, more precisely its unique connected level set, such that the density of the $\Gamma$-limit  $\psi_{\eta}=\psihom$ associated to $\varphi$, satisfies equation (\ref{eq psi eta}).


The function $\psi$ is characterized by
the convex and symmetric sub-level set
\[
C_{\psi}=\left\{ w \in \Rn{2}:\  \psi(w) \le  1 \right\}.
\]
Note that, by condition $\psi(w)\ge  |w|^2$, we have that $C_{\psi} \subseteq B_1(0)$, and, by the symmetry, $C_{\psi}$ is centered at the origin.

For every $N\in{\bf N}$, we can approximate this convex set with a polygon of $2N$ vertices, (which we still choose symmetric) $\pm V_1,\dots,\pm V_N$, whose directions are $\pm\nu_1, \dots,\pm\nu_N$. By density, we can also assume that these vertices are ‘‘rational'', in the sense that, for each $i=1,\dots,N$ there exists a point $z_i \in \mathbb{Z}^2$ and $t_i \in \mathbb{R}$, such that $t_i V_i = z_i$. For a pictorial description, we refer to Fig.~\ref{fig:Cpsi}
\begin{figure}[h!]\centering
\includegraphics[width=0.4\columnwidth]{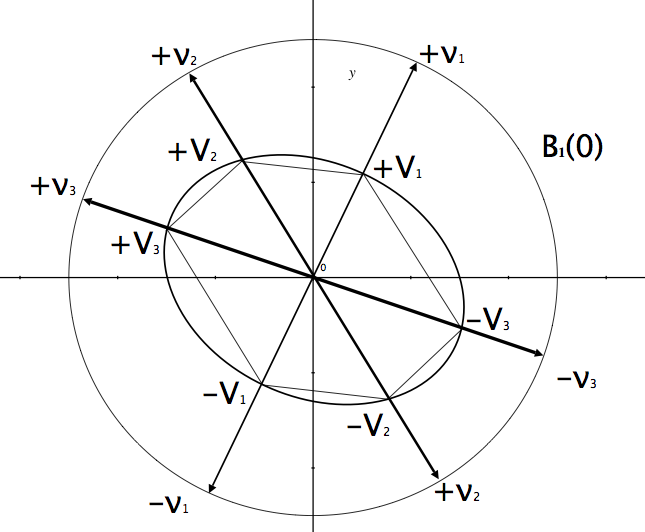}
\caption{The set $C_{\psi}$ and its polygonal approximation with directions $\nu_i$.}\label{fig:Cpsi}
\end{figure}

We will define $\varphi$ by constructing its unique connected level set $\{\varphi=0\}$, in such a way that the corresponding $\psihom$ has the polygon defined above as sublevel set. This will prove the approximation result.

From now on we directly assume that the target $\psi$ has polygonal level sets as above.
Let $Q=[0,a]^2$ be the periodicity square for all directions $\nu_i$; i.e., the square of edge the least common multiple $\tau=\text{lcm}(t_1,\dots,t_N)$.

 \begin{figure}[!h]\centering \includegraphics[width=0.5\columnwidth]{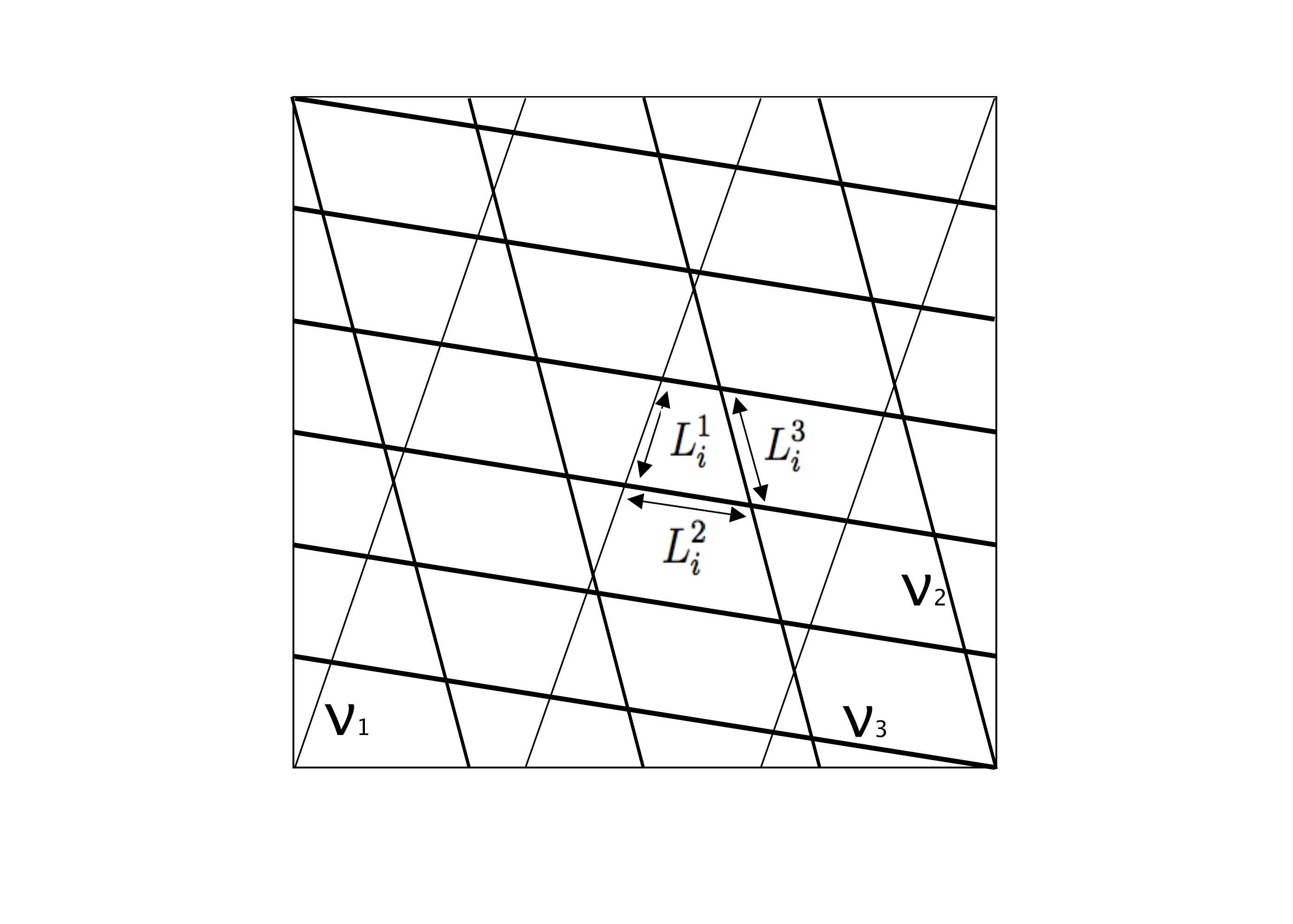}
\caption{The Q square with lines in directions $\nu_i$ and segments of length $L^j_i$.}\label{fig:Q}
\end{figure}
We begin our construction by considering the $Q$-periodic network of lines
$$
S=a\mathbb{Z}^2+\bigcup_{j=1}^N \nu_j\Rn{}
$$
This network defines a collection of segments $S^j_i\subset Q$,
 for $i=1,\dots,N$, $j=1,\dots,M_i$, with
$$
S^j_i\subset a\mathbb{Z}^2+ \nu_i\Rn{}
$$
and endpoints in
$$
\partial Q\cup \bigcup_{k\neq i}(a\mathbb{Z}^2+ \nu_k\Rn{})\cap
(a\mathbb{Z}^2+ \nu_i\Rn{}).
$$
We denote by $L^j_i$ the length of the segment $S^j_i$ (see Fig.~\ref{fig:Q}).
%
%
\begin{figure}[!h]\centering
\includegraphics[width=0.5\columnwidth]{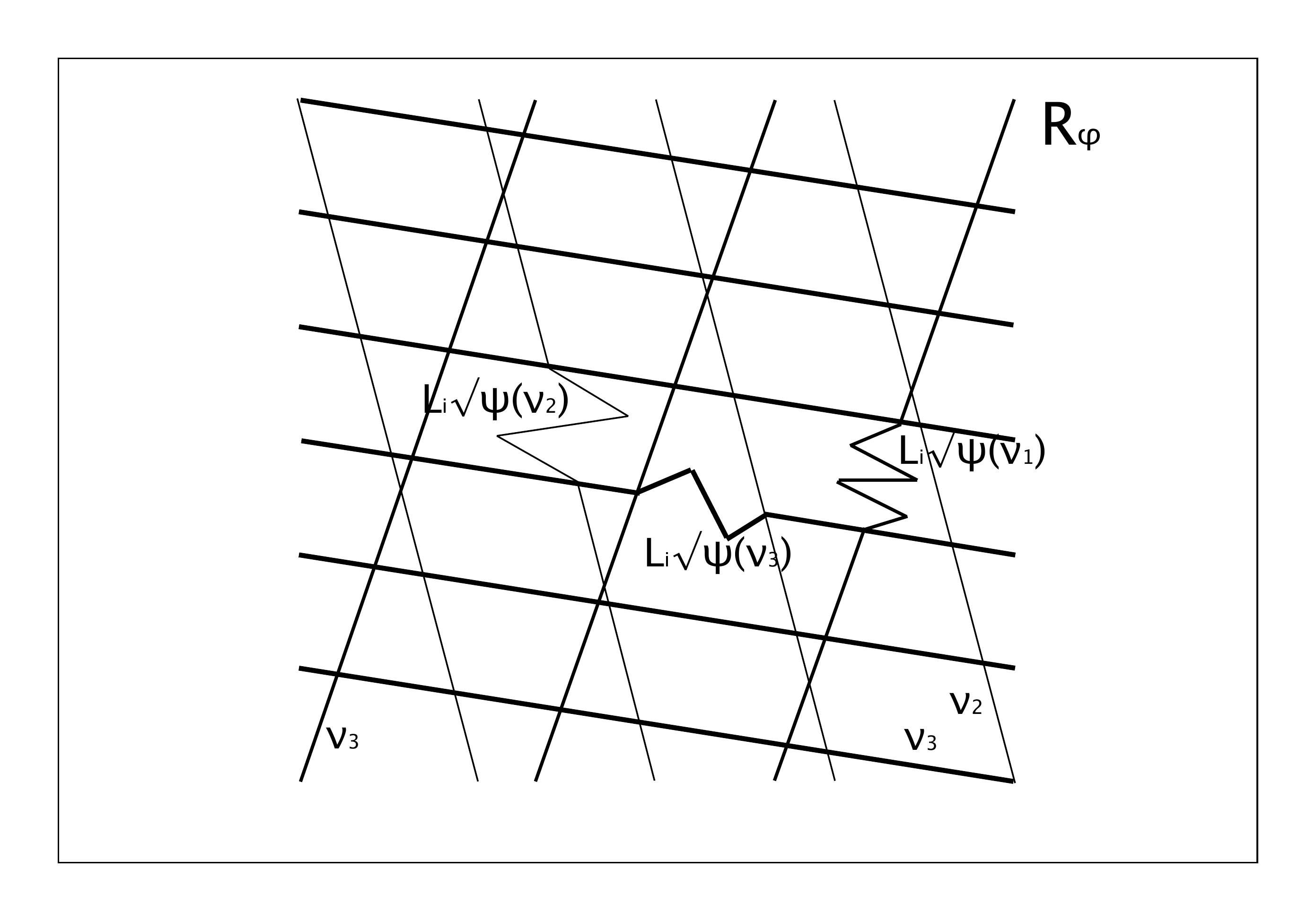}
\caption{The construction of the set $R_{\varphi}$, with three segments modified.}\label{fig:Rphi}
\end{figure}Note that if we took this network $S$ as the level set $\{\varphi=0\}$ then
we would have $\psihom(\nu_i)=1$ for all $i$. Since we want instead that
$$
\psihom(\nu_i)=\psihom\Bigl({V_i\over|V_i| }\Bigr)={1\over|V_i|^2 }\psihom({V_i})={1\over|V_i|^2 }
$$
we replace each segment  $S^j_i$, with a curve $\tilde S^j_i$ of length $L^j_i\sqrt{\psi(\nu_i)}$, as in Fig.~\ref{fig:Rphi}.

The set of all these lines $\tilde S^j_i$ inside $Q$, extended by periodicity, will represent the level set of $\varphi$ connecting $\Rn{2}$, i.e., by our assumption, the set $\left\{ \varphi=0 \right\}$, that we name $R_{\varphi}$. For example, we can take as $\varphi$ the squared distance from $R_{\varphi}$. Observe that, for such a constraint, the level sets are the union of $\mathcal{C}^1$ sets, so that hypothesis in Definition \ref{esistenza raccordi} is satisfied. The unique level made by a single unbounded  connected component is $\{\varphi=0 \}=R_{\varphi}$, all the other levels are made by infinitely many bounded connected components; i.e., the hypothesis in Definition \ref{HyH} is satisfied. Finally hypothesis in Definition \ref{hp geom} can be proved directly, arguing as in Example \ref{sinx siny}. Therefore, we may apply Theorem \ref{thm gammalim} to obtain a limit energy density $\psihom$.


Note that for any $i=1,\dots,N$
\begin{eqnarray*}
\psihom(\nu_i)&= &\lim_{T \to \infty} \frac{1}{T} \min \Bigl\{ \int_0^T |u'|^2dt,\\
&& |u(0)| \le  \sqrt{2}, |u(T)-T\nu_i| \le  \sqrt{2}, u \in R_{\varphi} \Bigr\}.
\end{eqnarray*}
We can test this formula with a function $u_T$, with $u_T(t)\in \bigcup_j\tilde S^j_i$ (i.e., taking its values in the deformation of a line in direction $\nu_i$) for all $t$ and with constant velocity.
%
Observe that, by the periodicity of $R_{\varphi}$, the distance covered by $u_T$ is at most
$
\sqrt{\psi(\nu_i)} \left[ T+1\right]$,
so that
\[
\psihom(\nu_i) \le  \lim_{T\to +\infty} \frac{1}{T} \int_0^T |u'|^2 dt = \lim_{T\to +\infty} \psi(\nu_i) \frac{1}{T^2} \left[ T+1\right]^2 = \psi(\nu_i).
\]
By convexity, we can extend the result to any $w \in \Rn{2}$, obtaining
\begin{equation}\label{le}
\psihom(w) \le  \psi(w)\qquad\hbox{ for all } w\in\Rn{2}.
\end{equation}

We now prove the converse inequality by estimating $\psi_T^0(w)$ from below. It is not restrictive to suppose that the test functions satisfy $v(0)=0$,
$v(T)=Tw\in a\mathbb{Z}^2$. Note moreover that by the convexity of $|v'|^2$ such a minimizer has constant velocity $|v'|=c$.
%
Denote by $\lambda_i$ the vector sum of all the segments $S^j_i$ in the image of $v$ in the direction $\nu_i$ (without the modification made by $\sqrt{\psi(\nu_i)}$).

Note that, a priori, $u$ may pass through some segments in the same direction $\nu_i$ but with opposite sign; in this case, will not consider the $\lambda_i$ related to these two portions of space. Therefore, in general, the distance covered by $v$, will be greater or equal then $\sum_{i=1}^N \lambda_i \sqrt{\psi(\nu_i)}$.

Note that, by construction, we have
\[
\sum_{i=1}^N \lambda_i \nu_i =Tw.
\]
Hence, for $T$ sufficiently large,
\begin{eqnarray*}
T \psihom(w) &=& \int_0^T |v'|^2 dt= T c^2 \\
&\ge & T \left(\frac{\sum_{i=1}^N \lambda_i \sqrt{\psi(\nu_i)}}{T}\right)^2
\\
&=&\frac{1}{T} \left( \sum_{i=1}^N \lambda_i \sqrt{\psi(\nu_i)} \right)^2 \frac{\left( \sum_{i=1}^N \lambda_i \right)^2}{\left( \sum_{i=1}^N \lambda_i \right)^2} \\
&=& \frac{\left( \sum_{i=1}^N \lambda_i \right)^2}{T} \left( \sum_{i=1}^N \frac{\lambda_i}{\sum_i \lambda_i} \sqrt{\psi(\nu_i)} \right)^2.
\end{eqnarray*}
The last term is a convex combination of $\sqrt{\psi(\nu_1)},\ldots, \sqrt{\psi(\nu_N)}$; then, by the convexity of $\sqrt{\psi}$ and  by the $2$-homogeneity of $\psi$, we get
\begin{eqnarray*}
T \psihom(w) &\ge &  \frac{\left( \sum_{i=1}^N \lambda_i \right)^2}{T} \left( \sqrt{\psi\left( \sum_{i=1}^N \frac{\lambda_i}{\sum_i \lambda_i} \nu_i \right)} \right)^2
\\
&=&\frac{\left( \sum_{i=1}^N \lambda_i \right)^2}{T} \cdot\frac{1}{\left( \sum_{i=1}^N \lambda_i \right)^2} \psi \left( \sum_{i=1}^N \lambda_i \nu_i \right)\\
&=&  \frac{1}{T} \psi(Tw)= T \psi(w).
\end{eqnarray*}
Therefore, by the above inequality and (\ref{le}), we get $
\psihom(w)=\psi(w)$,
as desired.

%

\medskip\noindent
{\bf Acknowledgments.} The first author acknowledges the hospitality of the Mathematical Institute in Oxford, where part of this work was written.


\begin{thebibliography}{22}
\bibitem{A-B-CP}
N. Ansini, A. Braides, and V. Chiad\`o Piat. Homogenization of periodic multi-dimensional structures. {\em Boll. Unione Mat. Ital.} (8) {\bf 2} (1999), 735--758.

\bibitem{BM}
J. F. Babadjian and V. Millot. Homogenization of variational problems in manifold valued BV-spaces. {\em Calc. Var. Partial Differential Equations} {\bf 36} (2009), 7--47.

\bibitem{BM2}
J. F. Babadjian and V. Millot. Homogenization of variational problems in manifold valued Sobolev spaces. {\em ESAIM Control Optim. Calc. Var.} {\bf 16} (2010), 833--855.

\bibitem{GCB}A. Braides.
{\it $\Gamma$-convergence for Beginners}.
Oxford University Press,  Oxford, 2002.

\bibitem{HB}A. Braides.
A handbook of $\Gamma$-convergence.
In {\it Handbook of Differential Equations.
Stationary Partial Differential Equations, Volume $3$}
(M. Chipot and P. Quittner, eds.), Elsevier, 2006, pp. 101--213.

\bibitem{B-B-F} A. Braides, G. Buttazzo and I. Fragal\`a.
Riemannian approximation of Finsler metrics.
{\it Asympt. Anal.} {\bf 31} (2002), 177--187

\bibitem{BC} A. Braides and V. Chiad\`o Piat.
Non convex homogenization problems for singular structures.
{\it Netw. Heterog. Media}, {\bf 3} (2008), 489--508.

\bibitem{BDF}A. Braides and A. Defranceschi.
{\it Homogenization of Multiple Integrals}.
Oxford University Press,  Oxford, 1998.

\bibitem{BF}
G. Bouchitt\'e and I. Fragal\`a. Homogenization of thin structures by two-scale method with respect to measures. {\em SIAM J. Math. Anal.} {\bf 32} (2001), 1198--1226.

\bibitem{Bu} D. Burago.
Periodic metrics.
In {\em Seminar on Dynamical Systems} (S. Kuksin, V. Lazutkin, J. P\"oschel, eds).
{\em Progress in Nonlinear Differential Equations and Their Applications} {\bf 12},
Birkh\"auser, Basel, 1994, pp. 90-95.
\bibitem{DM} G. Dal Maso. {\it An Introduction to $\Gamma$-convergence},
Birkhauser, Boston, 1993.

\bibitem{P-Z06}
A. L. Piatnitski and V. V. Zhikov. Homogenization of random singular structures and random measures. {\em Izv. Ross. Akad. Nauk Ser. Mat.} {\bf 70} (2006), 23--74, translation in {\em Izv. Math.} {\bf 70} (2006), 19--67.

\bibitem{R04}
A. P. Rybalko. Homogenization of harmonic 1-forms on pseudo Riemannian manifolds of complex microstructure. {\em Mat. Fiz. Anal. Geom.} {\bf 11} (2004), 249--257.

\bibitem{Z}
V. V. Zhikov. Averaging of problems in the theory of elasticity on singular structures. {\em Dokl. Akad. Nauk} {\bf 380} (2001), 741--745.
\end{thebibliography}
\end{document}